\documentclass[11pt]{amsart}

\usepackage{epigamath-voisin}

\usepackage[english]{babel}

\numberwithin{equation}{section}

\usepackage{enumitem}
\setlist[enumerate,1]{label={\rm(\roman*)}, ref={\rm\roman*}} 

\usepackage{tikz}\usetikzlibrary{decorations.markings,matrix,arrows}
\usepackage{tikz-cd}
\usepackage{amsfonts}
\usepackage{amsthm}
\usepackage{cancel}
\usepackage[T1]{fontenc}
\usepackage[latin9]{inputenc}
\usepackage[mathscr]{eucal}
\usepackage{amssymb,latexsym,amsmath,amscd}
\usepackage{graphicx}
\usepackage{calrsfs}
\usepackage{mathtools}
\usepackage{titlesec}
\usepackage{todonotes}
\usepackage{scalerel,stackengine}
\usepackage{braket}
\usepackage{relsize}
\usepackage{cleveref}

\allowdisplaybreaks

\theoremstyle{plain}
\newtheorem{thmm}{Theorem}[section]

\newtheorem{thm}[thmm]{Theorem}
\newtheorem*{thm*}{Theorem}

\newtheorem{lem}[thmm]{Lemma}

\newtheorem{lem-def}[thmm]{Lemma-Definition}
\newtheorem*{claim}{Claim}
\newtheorem{pro}[thmm]{Proposition}

\newtheorem{pro-def}[thmm]{Proposition-Definition}
\newtheorem{cor}[thmm]{Corollary}

\newtheorem{conj}[thmm]{Conjecture}
\newtheorem{que}[thmm]{Question}

\newtheorem{prob}[thmm]{Problem}

\theoremstyle{definition}

\theoremstyle{remark}
\newtheorem{rem}[thmm]{Remark}

\newcommand{\ssec}{\subsection}

\newcommand{\ol}{\overline}
\newcommand{\ti}[1]{\tilde{#1}}
\newcommand{\ul}{\underline}
\newcommand{\vast}{\bBigg@{4}}
\newcommand{\Vast}{\bBigg@{5}}
\newcommand{\wt}{\widetilde}

\newcommand{\wh}{\widehat}

\stackMath
\newcommand\reallywidehat[1]{%
\savestack{\tmpbox}{\stretchto{%
  \scaleto{%
    \scalerel*[\widthof{\ensuremath{#1}}]{\kern-.6pt\bigwedge\kern-.6pt}%
    {\rule[-\textheight/2]{1ex}{\textheight}}
  }{\textheight}%
}{0.5ex}}%
\stackon[1pt]{#1}{\tmpbox}%
}

\definecolor{armygreen}{rgb}{0.29, 0.33, 0.13}
\definecolor{ao(english)}{rgb}{0.5, 0.2, 0.0}

\newcommand{\bC}{\mathbf{C}}

\newcommand{\bH}{\mathbf{H}}

\newcommand{\bP}{\mathbf{P}}
\newcommand{\bQ}{\mathbf{Q}}
\newcommand{\bR}{\mathbf{R}}

\newcommand{\bW}{\mathbf{W}}
\newcommand{\bZ}{\mathbf{Z}}

\newcommand{\bm}{\mathbf{m}}

\newcommand{\cB}{\mathcal{B}}
\newcommand{\cC}{\mathcal{C}}

\newcommand{\cE}{\mathcal{E}}
\newcommand{\cF}{\mathcal{F}}

\newcommand{\cH}{\mathcal{H}}

\newcommand{\cJ}{\mathcal{J}}
\newcommand{\cK}{\mathcal{K}}
\newcommand{\cL}{\mathcal{L}}
\newcommand{\cM}{\mathcal{M}}

\newcommand{\cO}{\mathcal{O}}

\newcommand{\sC}{\mathscr{C}}

\newcommand{\sX}{\mathscr{X}}
\newcommand{\sY}{\mathscr{Y}}

\newcommand{\gO}{\Omega}
\newcommand{\gS}{\Sigma}

\newcommand{\ga}{\alpha}
\newcommand{\gb}{\beta}
\newcommand{\gd}{\delta}

\newcommand{\gl}{\lambda}
\newcommand{\go}{\omega}
\newcommand{\gs}{\sigma}

\newcommand{\Alb}{\mathrm{Alb}}
\newcommand{\alb}{\mathrm{alb}}

\newcommand{\BM}{\mathrm{BM}}

\newcommand{\cl}{\mathrm{cl}}

\newcommand{\coker}{\mathrm{coker}}

\newcommand{\Ima}{\mathrm{Im}}
\newcommand{\Int}{\mathrm{Int}}

\newcommand{\NS}{\mathrm{NS}}

\newcommand{\PD}{\mathrm{PD}}

\newcommand{\PH}{\mathrm{PH}}

\newcommand{\Psef}{\mathrm{Psef}}

\newcommand{\Sing}{\mathrm{Sing}}

\newcommand{\supp}{\mathrm{supp}}

\newcommand{\cf}{\textit{cf.}}
\newcommand{\eg}{\textit{e.g.}}

\newcommand{\ie}{\textit{i.e.}}

\newcommand{\bss}{\backslash}

\newcommand{\cnec}{\mathrel{:=}}

\newcommand{\dr}{\partial}
\newcommand{\drbar}{\bar{\partial}}
\newcommand{\ddbar}{\dr\bar{\dr}}
\newcommand{\st}{\star}

\renewcommand{\(}{\left(}
\renewcommand{\)}{\right)}

\newcommand{\la}{\mathopen{\langle}}
\newcommand{\ra}{\mathclose{\rangle}}

\newcommand{\dto}{\dashrightarrow}
\newcommand{\lto}{\leftarrow}

\newcommand{\hto}{\hookrightarrow}
\newcommand{\xto}[1]{\xrightarrow{ #1 }}
\newcommand{\xhto}[1]{\xhookrightarrow{ #1 }}
\newcommand{\xlto}[1]{\xleftarrow{ #1 }}

\def\isolow{\vbox to 0pt{\vss\hbox{$\sim$}\vskip-3pt}}
\newcommand{\eto}{\xrightarrow{\isolow}}
\newcommand{\elra}{\xrightarrow{\;\isolow\;}}

\makeatletter
\let\orgdescriptionlabel\descriptionlabel
\renewcommand*{\descriptionlabel}[1]{%
  \let\orglabel\label
  \let\label\@gobble
  \phantomsection
  \edef\@currentlabel{#1}%
  \let\label\orglabel
  \orgdescriptionlabel{#1}%
}
\makeatother
\tikzset{node distance=2cm, auto}

\newcommand{\lra}{\longrightarrow}

\newcommand{\supth}[1]{\ensuremath{#1^{\mathrm{th}}}}

\YearArticle{2024} \EpigaArticleNr{13} \ReceivedOn{January 5, 2023}
\InFinalFormOn{September 30, 2023}
\AcceptedOn{October 17, 2023}

\title{On the dual positive cones and the algebraicity of compact K\"ahler manifolds}
\titlemark{On the dual positive cones and the algebraicity of compact K\"ahler manifolds}

\author{Hsueh-Yung Lin}
\address{Department of Mathematics, National Taiwan University, 
and National Center for Theoretical Sciences,
Taipei, Taiwan}
\email{hsuehyunglin@ntu.edu.tw}

\authormark{H.-Y.~Lin}

\AbstractInEnglish{We investigate the algebraicity of compact K\"ahler manifolds admitting a positive rational Hodge class of bidimension $(1,1)$.  We prove that if the dual K\"ahler cone of a compact K\"ahler manifold $X$ contains a rational class as an interior point, then its Albanese variety is projective.  As a consequence, we answer the Oguiso--Peternell problem for Ricci-flat compact K\"ahler manifolds.  We also study related algebraicity problems for threefolds. }


\acknowledgement{The author is supported by the Taiwan Ministry of Education Yushan Young Scholar Fellowship (NTU-110VV006) and the National Science and Technology Council (110-2628-M-002-006-). }

\dedication{Pour Claire Voisin, \`a l'occasion de son soixanti\`eme anniversaire}

\begin{document}


\maketitle

\begin{prelims}

\DisplayAbstractInEnglish

\bigskip

\DisplayKeyWords

\medskip

\DisplayMSCclass

\end{prelims}


\newpage

\setcounter{tocdepth}{1}

\tableofcontents


\section{Introduction}

\ssec{Dual statements of the Kodaira embedding theorem}

Let $X$ be a compact K\"ahler manifold of dimension $n$. The celebrated Kodaira embedding theorem asserts that if the K\"ahler cone $\cK(X)$ of $X$ contains a rational cohomology class, then $X$ is projective; \cf~\cite{Kodaira1954}. Now consider the \textit{dual K\"ahler cone}
$$\cK(X)^\vee = \left\{ \ \ga \in H^{n-1,n-1}(X,\bR) \cnec H^{n-1,n-1}(X) \cap H^{2n-2}(X,\bR) \ \big{ | } \ \la \ga,\go \ra \ge 0 \text{ for every } \go \in \cK(X) \ \right\} $$
of $X$, where
$$\la \ \ \ , \ \ \ \ra \colon H^{n-1,n-1}(X,\bR) \otimes H^{1,1}(X,\bR) \lra \bR$$
is the perfect pairing defined by  Poincar\'e duality.  As the closure $\ol{\cK(X)}$ of $\cK(X)$ in $H^{1,1}(X,\bR)$ is a salient cone, the dual $\cK(X)^\vee$ has non-empty interior $\Int\(\cK(X)^\vee\)$ in $H^{n-1,n-1}(X,\bR)$.  The following problem was first asked and studied by Oguiso and Peternell~\cite{OguisoPeternellconeKahduasurf, OguisoPeternellconeKahdual} in search of a dual statement of the Kodaira embedding theorem.

\begin{prob}[Oguiso--Peternell]\label{prob-OP}
Let $X$ be a compact K\"ahler manifold of dimension $n$ such that $\Int\(\cK(X)^\vee\)$ contains an element of $H^{2n-2}(X,\bQ)$.  How algebraic is $X$? For instance, what are the possible algebraic dimensions of $X$?
\end{prob}

Here is another problem dual to the Kodaira embedding theorem that we can formulate.  The closure $\ol{\cK(X)}$ of the K\"ahler cone is called the \textit{nef} cone of $X$.  Thanks to Demailly and P\u{a}un, we know that the Poincar\'e dual of $\ol{\cK(X)}$ is the closed convex cone in $H^{n-1,n-1}(X,\bR)$ generated by the classes of closed positive currents of type $(n-1,n-1)$; \cf~\cite[Theorem 2.1]{BDPP}.  The analog of such a cone in $H^{1,1}(X,\bR)$ is the \textit{pseudoeffective cone} $\Psef(X)$, defined as the closed convex cone in $H^{1,1}(X,\bR)$ generated by the classes of closed positive currents of type $(1,1)$. From this point of view, if $\Psef(X)^\vee \subset H^{n-1,n-1}(X,\bR)$ denotes the Poincar\'e dual of $\Psef(X)$ and  $\Int\(\Psef(X)^\vee\)$  its interior in $H^{n-1,n-1}(X,\bR)$, then the following can also be considered as a dual problem to the Kodaira embedding theorem.

\begin{prob}\label{prob-psefdual}
Let $X$ be a compact K\"ahler manifold of dimension $n$ such that $\Int\(\Psef(X)^\vee\)$ contains an element of $H^{2n-2}(X,\bQ)$. Is $X$ always projective? If not, how algebraic is $X$?
\end{prob}

We say that a compact K\"ahler manifold $X$ of dimension $n$ satisfies the \textit{dual Kodaira condition (K)} if $\Int\(\cK(X)^\vee\)$ contains an element of $H^{2n-2}(X,\bQ)$.  The \textit{dual Kodaira condition (P)} is defined similarly with $\Int\(\cK(X)^\vee\)$ replaced by $\Int\(\Psef(X)^\vee\)$.

We can compare the dual Kodaira conditions (K) and (P), together with the condition in the Kodaira embedding theorem, as follows. On the one hand if $\go \in \cK(X)$, then $\go^{n-1} \in \Int\(\Psef(X)^\vee\)$, so condition (P) is weaker than the condition in the Kodaira embedding theorem.  On the other hand since $\cK(X) \subset \Psef(X)$, condition (P) is stronger than condition (K).  It is still unknown whether these conditions are equivalent.

In~\cite{OguisoPeternellconeKahdual}, a similar problem had also been studied by Oguiso and Peternell. 

\begin{prob}[Oguiso--Peternell]\label{prob-OPCampl}
Let $X$ be a compact K\"ahler manifold of dimension $n$ such that $X$ contains a smooth curve with ample normal bundle. How algebraic is $X$?
\end{prob}

The two Oguiso--Peternell problems (Problems~\ref{prob-OP} and~\ref{prob-OPCampl}) could be related by the conjecture that if $C \subset X$ is a smooth curve with ample normal bundle, then $[C] \in \Int\(\cK(X)^\vee\)$; \cf~\cite[Conjecture~0.3]{OguisoPeternellconeKahdual}.  All the problems introduced above aim at understanding the algebraicity of compact K\"ahler manifolds containing some positive rational Hodge class of \textit{bidimension} $(1,1)$.

\ssec{Main results}

In this article, we study the aforementioned problems.  Our first result provides the following partial answer to the Oguiso--Peternell problem (and also to Problem~\ref{prob-psefdual}).

\begin{thm}\label{thm-alb}
Let $X$ be a compact K\"ahler manifold.  If\, $X$ satisfies the dual Kodaira condition (K), then the Albanese torus of $X$ is projective.
\end{thm}

For a compact K\"ahler manifold $X$ as in Problem~\ref{prob-OP}, Theorem~\ref{thm-alb} gives a lower bound of the algebraic dimension $a(X)$ of $X$, namely the dimension of its Albanese image.  In particular, if $X$ has maximal Albanese dimension (namely, if the Albanese map is generically finite onto its image), then $X$ is projective.
 
 \begin{rem}
If we replace condition (K) in Theorem~\ref{thm-alb} with the existence of a smooth curve $C \subset X$ such that $N_{C/X}$ is ample (hence in the context of Problem~\ref{prob-OPCampl}), then the projectivity of $\Alb(X)$ simply follows from the surjectivity of $\Alb(C) \to \Alb(X)$; \cf~\cite[Lemma 12]{Ottem2016}.
 \end{rem}

The essential part of the proof of Theorem~\ref{thm-alb} consists in answering Problem~\ref{prob-OP} for complex tori; see Proposition~\ref{pro-T} for the precise statement.  As a consequence of Proposition~\ref{pro-T} (together with Proposition~\ref{pro-HK} about hyper-K\"ahler manifolds), we also answer Problems~\ref{prob-OP} for Ricci-flat compact K\"ahler manifolds.

\begin{thm}\label{thm-RP} 
Let $X$ be a compact K\"ahler manifold with $c_1(X) = 0 \in H^2(X,\bR)$.  If\, $X$ satisfies the dual Kodaira condition (K), then $X$ is projective.
\end{thm}

For a compact K\"ahler surface $S$, Huybrechts~\cite{HuybCHKbasic, Huybrechts2003} and independently Oguiso--Peternell~\cite{OguisoPeternellconeKahduasurf} proved that if $S$ satisfies the dual Kodaira condition (K), then $S$ is projective. This completely answers Problem~\ref{prob-OP} and Problem~\ref{prob-psefdual} in dimension 2.  In this article, we answer Problem~\ref{prob-psefdual} in dimension 3 except for simple non-Kummer threefolds, which presumably do not exist (see Remark~\ref{rem-simpNK}).

\begin{thm}\label{thm-mainOP}
Let $X$ be a smooth compact K\"ahler threefold.  Assume that $X$ is not a simple non-Kummer threefold.  If\, $X$ satisfies the dual Kodaira condition (P), then $X$ is projective.
\end{thm}

The study of the Oguiso--Peternell problem in the threefold case was initiated by Oguiso and Peternell in~\cite{OguisoPeternellconeKahdual}.  Suppose that $X$ is a compact K\"ahler threefold which is not simple non-Kummer.  They proved that if $\Int\(\cK(X)^\vee\)$ contains an element of $H^{2n-2}(X,\bQ)$ which is a curve class, then $X$ has algebraic dimension $a(X) \ge 2$.  We improve their result by removing the curve class assumption.

\begin{thm}\label{thm-mainOPK}
Let $X$ be a smooth compact K\"ahler threefold.  Assume that $X$ is not a simple non-Kummer threefold.  If\, $X$ satisfies the dual Kodaira condition (K), then $a(X) \ge 2$.
\end{thm}

\begin{rem}\label{rem-simpNK}
In both Theorems~\ref{thm-mainOP} and~\ref{thm-mainOPK}, the existence of simple non-Kummer threefolds would be excluded if the abundance conjecture holds for compact K\"ahler threefolds (see~\cite[Proof of Theorem 6.2]{HorPetsurvey}).  This is proven in a very recent paper by Das and Ou~\cite{das2023log}.
\end{rem}

If we replace the dual Kodaira condition in Theorem~\ref{thm-mainOPK} with the existence of a smooth curve in $X$ with ample normal bundle (thus in the context of Problem~\ref{prob-OPCampl}), then the same conclusion $a(X) \ge 2$ holds; this was proven by Oguiso and Peternell~\cite[Theorem 0.5]{OguisoPeternellconeKahdual}.  Oguiso and Peternell have also outlined a strategy to construct a non-algebraic threefold $X$ such that $\Int\(\cK(X)^\vee\)$ contains a curve class (or such that $X$ contains a smooth curve with ample normal bundle). An explicit construction of such an example is still missing.

While we are still unable to solve Problems~\ref{prob-OP}
and~\ref{prob-OPCampl} in dimension 3, we can relate these problems to
a problem about 1-cycles in threefolds. We postpone the discussion to
Section~\ref{ssec-1cycles}, after we give an outline of the proofs of
Theorems~\ref{thm-mainOP} and~\ref{thm-mainOPK} about
threefolds.

\subsection{Outline of the proofs of Theorems~\ref{thm-mainOP} and~\ref{thm-mainOPK}}

To prove Theorem~\ref{thm-mainOPK}, we have to show that if $X$ is a compact K\"ahler threefold $X$ of algebraic dimension $a(X) \le 1$ which is not simple non-Kummer, then $\Int\(\cK(X)^\vee\) \cap H^{2n-2}(X,\bQ) = \emptyset$.  Essentially based on Fujiki's descriptions of algebraic reductions of compact K\"ahler threefolds~\cite{Fujiki1983}, those threefolds are bimeromorphic to one of the following (see Proposition~\ref{pro-class}):
\begin{enumerate}
\item\label{threef-1} a threefold which dominates a surface, 
\item\label{threef-2} a fibration in abelian varieties over a curve, 
\item\label{threef-3} a finite quotient of a smooth isotrivial torus fibration over a curve without multi-sections,
\item\label{threef-4} a finite quotient of a 3-torus $T$.
\end{enumerate}

According to the bimeromorphic invariance of the emptiness of $\Int\(\cK(X)^\vee\) \cap H^{2n-2}(X,\bQ)$ (see Proposition~\ref{pro-domduK}), it suffices to prove that $\Int\(\cK(X)^\vee\) \cap H^{2n-2}(X,\bQ) = \emptyset$ for the above four types of varieties.  The first case is an immediate consequence of~\cite[Proposition 2.6]{OguisoPeternellconeKahdual}.  The proof in cases~\eqref{threef-2},~\eqref{threef-3},  and~\eqref{threef-3} will be carried out in Sections~\ref{sec-fibtlis}, \ref{sec-fibab}, and \ref{sec-tores}, respectively.

Since $\cK(X) \subset \Psef(X)$, Theorem~\ref{thm-mainOPK} also implies that a compact K\"ahler threefold $X$ as in Theorem~\ref{thm-mainOP} satisfies $a(X) \ge 2$.  Therefore, to prove Theorem~\ref{thm-mainOP} for $X$, it suffices to exclude the possibility that $a(X) = 2$.  Compact K\"ahler threefolds of algebraic dimension 2 are bimeromorphic to elliptic fibrations over a projective surface, and we will prove Theorem~\ref{thm-mainOP} for elliptic threefolds in Section~\ref{sec-fibellip}.

\ssec{A question about 1-cycles and the Oguiso--Peternell problem}\label{ssec-1cycles}
 
We now relate Problem~\ref{prob-OP} to a question about 1-cycles in compact K\"ahler threefolds, which we first formulate.
 
Let $X$ be a compact K\"ahler manifold and $Y \subset X$ a complex subvariety of codimension $l$. Let $\ga \in H^{k,k}(X,\bQ)$ be a Hodge class which vanishes in $H^{2k}(X\bss Y,\bQ)$. Since $H^{2k}(X,\bQ)$ carries a pure Hodge structure, $\ga$ belongs to the image of $\imath_* \colon H^{2k-2l}(\ti{Y},\bQ) \to H^{2k}(X,\bQ)$, where $\imath \colon \ti{Y} \to X$ is the composition of a K\"ahler desingularization $\ti{Y} \to Y$ of $Y$ with $Y \hto X$ (see \eg~Lemma~\ref{lem-excision}).  If we moreover assume that $X$ is projective, then based on the existence of polarization on the underlying $\bQ$-Hodge structure of (the summands of the primitive decomposition of) $H^{2k}(X,\bQ)$, the class $\ga$ is even the image of a Hodge class $\gb \in H^{k-l,k-l}(\ti{Y},\bQ)$; \cf~\cite[Remark 2.30]{Voisinchowdecomp}.  Without the projectivity assumption on $X$, it is yet unknown whether the latter property still holds, even for 1-cycles in threefolds.
 
\begin{que}\label{que-1cycles}
Let $X$ be a smooth compact K\"ahler threefold, and let $Y \subset X$ be a surface $($possibly singular with several irreducible components$)$.  Let $\imath \colon \ti{Y} \to X$ be the composition of a desingularization $\ti{Y} \to Y$ of\, $Y$ with the inclusion $Y \hto X$.  Given a Hodge class $\ga \in H^{2,2}(X,\bQ)$ which vanishes in $H^4(X\bss Y,\bQ)$, does there exist a $\gb \in H^{1,1}(\ti{Y},\bQ)$ such that $\imath_*\gb = \ga$?
\end{que}

So far, Question~\ref{que-1cycles} can be answered in the affirmative if $Y$ is irreducible (see Lemma~\ref{lem-ConjNKah3}).  Under the assumption that Question~\ref{que-1cycles} has a positive answer, we are able to answer Problems~\ref{prob-OP} and~\ref{prob-OPCampl} in dimension $3$ (except for simple non-Kummer threefolds; see Remark~\ref{rem-simpNK}).

\begin{cor}\label{cor-OPpar}
Assume that Question~\ref{que-1cycles} has a positive answer.  Let $X$ be a smooth compact K\"ahler threefold which is not simple non-Kummer.  If\, $X$ satisfies the conditions in Problems~\ref{prob-OP} or~\ref{prob-OPCampl}, then $X$ is projective.
\end{cor}

\subsection{The existence of connecting families of curves and Problem~\ref{prob-psefdual}}

We finish this introduction by discussing how one could expect Problem~\ref{prob-psefdual} to have a positive answer. Let $X$ be a compact K\"ahler manifold.  The \textit{movable cone} $\cM(X) \subset H^{n-1,n-1}(X,\bR)$ of $X$ is defined as the closed convex cone generated by classes of the form $\mu_*(\go_1 \wedge \cdots \wedge \go_{n-1})$, where $\mu \colon \ti{X} \to X$ is a bimeromorphic morphism from a compact K\"ahler manifold $\ti{X}$ and the $\go_i$ are K\"ahler classes on $\ti{X}$.  Let
$$\cM(X)_{\NS} \cnec \cM(X) \cap \NS(X)_\bR,$$
where $\NS(X)_\bR$ denotes the $\bR$-span of $H^{n-1,n-1}(X,\bQ)$ in $H^{n-1,n-1}(X,\bR)$.  As a consequence of~\cite[Theorem 2.4]{BDPP}, if we assume that $X$ is projective, then $\cM(X)_{\NS}$ coincides with the closed convex cone generated by the classes of \textit{connecting families} of curves $(C_t)_{t \in T}$ (which means that for every general pair of points $x , y \in X$, there exist $t_1,\ldots,t_l \in T$ such that $x,y \in C_{t_1} \cup \cdots \cup C_{t_l}$ and $C_{t_1} \cup \cdots \cup C_{t_l}$ is connected).  One can ask whether this property remains true without the projectivity assumption.

\begin{que}\label{que-intmob}
Let $X$ be a compact K\"ahler manifold.  Does the N\'eron--Severi part $\cM(X)_{\NS}$ of the movable cone coincide with the closed convex cone generated by the classes of a connecting family of curves?
\end{que}

Conjecturally, $\cM(X)$ is the dual cone of $\Psef(X)$; \cf~\cite[Conjecture 2.3]{BDPP}. If we assume this conjecture and that Question~\ref{que-intmob} has a positive answer, then a compact K\"ahler manifold $X$ satisfying $\Int(\Psef(X)^\vee) \cap H^{n-1,n-1}(X,\bQ) \ne \emptyset$ would be algebraically connected (see Section~\ref{ssec-critCampana} for the definition).  By Campana's projectivity criterion (Theorem~\ref{thm-critcamp}), a compact K\"ahler manifold $X$ as in Problem~\ref{prob-psefdual} would then be projective.  Together with the evidence provided by threefolds (Theorem~\ref{thm-mainOP}) and Ricci-flat manifolds (Theorem~\ref{thm-RP}), this leads us to conjecture the following.

\begin{conj}
Let $X$ be a compact K\"ahler manifold of dimension $n$.  If\, $X$ satisfies the dual Kodaira condition (P), then $X$ is projective.
\end{conj}

\ssec{Organization of the article}

In the next section, we recall and prove some preliminary results that we need in this article.  In Section~\ref{sec-psefD}, we prove some results about the invariance of the conditions in Problems~\ref{prob-OP} and~\ref{prob-psefdual} under dominant maps.  In Section~\ref{sec-fibtlis}, we study smooth isotrivial torus fibrations, which will be useful to prove the main theorems.  We then prove Theorem~\ref{thm-alb} in Section~\ref{sec-tores} and Theorem~\ref{thm-RP} in Section~\ref{sec-Rp}.  In Section~\ref{sec-fibab}, we study fibrations in abelian varieties over a curve, similarly to what we do in Section~\ref{sec-fibtlis} for smooth isotrivial torus fibrations.  Section~\ref{sec-fibellip} is devoted to elliptic fibrations.  We will conclude the proofs of Theorems~\ref{thm-mainOP} and~\ref{thm-mainOPK} in Section~\ref{sec-preuve}.  Finally, we prove Corollary~\ref{cor-OPpar} in Section~\ref{sec-1cycles}, which relates the Oguiso--Peternell problems to Question~\ref{que-1cycles}.

\section*{Acknowledgments}

The author is indebted to Junyan Cao for bringing the Oguiso--Peternell problem to his attention, and for numerous discussions directly related to this work.  He also expresses warm thanks to Andreas H\"oring and Keiji Oguiso for stimulating discussions, as well as to the referee for valuable comments and questions.

\section{Preliminaries}\label{sec-prelim}
 
\ssec{Convention and terminology}

In this article, compact complex manifolds and varieties are assumed to be irreducible (but subvarieties can be reducible).  A \textit{fibration} is a surjective proper holomorphic map $f \colon X \to B$ with connected fibers.  We say that a compact complex variety $X$ is in the Fujiki class $\cC$ if $X$ is meromorphically dominated by a compact K\"ahler manifold.

Let $X$ be a compact complex manifold.  The hypercohomology of a bounded complex of sheaves of abelian groups $\cF^\bullet$ on $X$ is denoted by $\bH^\bullet(X,\cF^\bullet)$.  For any subring $R \subset \bC$, we define $H^{k,k}(X,R)$ as the kernel of the composition
$$H^{2k}(X,R) \lra H^{2k}(X,\bC) \simeq \bH^{2k}(X,\gO_X^{\bullet})
 \lra \bH^{2k}\(X,\gO_X^{\bullet \le k-1}\),$$
where $\gO_X^{\bullet}$ is the holomorphic de Rham complex.  Note that we have
$$H^{k,k}(X,\bQ) \simeq H^{k,k}(X,\bZ) \otimes_\bZ \bQ$$
(which does not hold in general if $\bQ$ is replaced by other rings, \eg~$\bR$).  When $X$ is a compact K\"ahler manifold, $H^{k,k}(X,R)$ is the sub-$R$-module of $H^{2k}(X,R)$ consisting of elements whose image in $H^{2k}(X,\bC)$ is of type $(k,k)$ with respect to the Hodge decomposition.  The weight filtration of a mixed Hodge structure $H$ is denoted by $\bW_\bullet H$.

We follow~\cite{DemaillyAM} for the definition and convention of positive cones, and the reader is referred to \textit{op.~cit.}\ for related basic properties.

\ssec{K\"ahler forms on normal complex spaces}

Let $X$ be a normal complex space (for instance, the quotient of a complex manifold by a finite group).  A smooth function on $X$ is a continuous function $f \colon X \to \bR$ such that for some open cover $\{U_i\}$ of $X$, there exist holomorphic embeddings $U_i \hto \bC^N$ such that each $f_{|U_i}$ extends to a smooth function on a neighborhood of $U_i$ in $\bC^N$.  The sheaf of germs of smooth functions on $X$ is denoted by $\cC_X^\infty$. Similarly, a strictly plurisubharmonic (or psh for short) function on $X$ is an upper semi-continuous function with values in $\bR \cup \{-\infty\}$ which extends to a strictly psh function on a neighborhood of a local embedding $X \hto \bC^N$.

Let $\PH_X = \Re \cO_X \subset \cC_X^\infty$ be the subsheaf of pluriharmonic functions (\ie, the real part of $\cO_X$).  A \textit{K\"ahler metric} on $X$ is a collection of smooth strictly psh functions $\Set{\phi_i \colon U_i \to \bR }_{i \in I}$, where $\{U_i\}_{i \in I}$ is an open cover of $X$ such that ${\phi_i}_{|U_i \cap U_j} - {\phi_j}_{|U_i \cap U_j} \in \PH_X(U_i \cap U_j)$.  In particular, a K\"ahler metric on $X$ is an element of $ H^0(X,\cC^\infty_X/ \PH_X)$.  The short exact sequences
\begin{equation}\label{SE-CPH}
0  \lra \PH_X  \lra \cC^\infty_X \lra \cC^\infty_X / \PH_X  \lra 0
\end{equation}
and
\begin{equation}\label{SE-PH}
0  \lra \bR  \xto{\times \sqrt{-1}}  \cO_X \xto{2 \cdot \Re}  \PH_X  \lra 0
\end{equation}
give the composition
 $$ [ \bullet ] \colon H^0(X,\cC_X^\infty/ \PH_X) \lra H^1(X, \PH_X) \lra H^2(X,\bR),$$
and a K\"ahler class $[\go] \in H^2(X, \bR)$ is the image of a K\"ahler metric $\go \in H^0(X,\cC^\infty/ \PH_X)$. The K\"ahler classes form a convex cone $\cK(X) \subset H^2(X,\bR)$, and elements in the closure $\ol{\cK(X)}$ are called \textit{nef} classes of $X$.

The following two lemmas are presumably well known; they are direct consequences of Varouchas' work~\cite{VarouchasKahdom, VarouchasKahsp}.

\begin{lem}\label{lem-pullbackKahG}
Let $f \colon X \to X/G$ be the quotient of a compact K\"ahler manifold $X$ by a finite group $G$. A $G$-invariant nef class $[\go] \in H^2(X,\bR)$ is the pullback $f^*[\go']$ of a nef class $[\go'] \in H^2(X/G,\bR)$.
\end{lem}

\begin{proof}
Since $f^* \colon H^2(X/G,\bR) \to H^2(X,\bR)$ is injective (because $f$ is finite) and nef classes are limits of K\"ahler classes, it suffices to prove that a $G$-invariant K\"ahler class $[\go] \in H^2(X,\bR)$ is the pullback $f^*[\go']$ of a K\"ahler class $[\go'] \in H^2(X/G,\bR)$.  Let $\go$ be a $G$-invariant K\"ahler form representing $[\go]$.  Then we can find an open cover $\{V_i\}$ of $X/G$ and $\cC^\infty$-strictly psh functions $u_i$ defined over $U_i \cnec f^{-1}(V_i)$ such that $\go_{|U_i} = i \dr \drbar u_i$.  By averaging over the $G$-orbits, we can assume that the functions $u_i$ are $G$-invariant.  By~\cite[Lemma II.3.1.2]{VarouchasKahsp}, the pushforwards $v_i = f_*u_i$ are still strictly psh,  and the $v_{ij} = (v_i - v_j)_{|V_i \cap V_j}$ are pluriharmonic. By~\cite[Theorem 1]{VarouchasKahsp},\footnote{Note that since $X/G$ is reduced, condition (ii) in~\cite[Theorem 1]{VarouchasKahsp} holds automatically; see~\cite[Remark II.2.2]{VarouchasKahsp}.} there exist $\cC^\infty$-strictly psh functions $v'_i$ defined over $V_i$ such that $v'_i - v'_j = v_{ij}$.  Therefore, the \v{C}ech 1-cocycle $\frac{1}{|G|}\{v_{ij}\}$ with coefficients in $\PH_{X/G}$ maps to a K\"ahler class $[\go'] \in H^2(X/G,\bR)$ under the connecting morphism of~\eqref{SE-PH}.  Since $\frac{1}{|G|}f^*v_{ij} = (u_i - u_j)_{|U_i \cap U_j}$, we have $f^*[\go'] = [\go]$.
\end{proof}

\begin{lem}\label{lem-pousKah}
Let $f \colon X \to Y$ be a finite morphism between compact complex manifolds.  If\, $[\go] \in H^2(X,\bR)$ is a K\"ahler class, then $f_*[\go] \in H^2(Y,\bR)$ is also a K\"ahler class.
\end{lem}

\begin{proof}
Let $\go$ be a K\"ahler form which represents $[\go]$.  For every $y \in Y$ and every $x \in f^{-1}(y)$, there exists a neighborhood $U(x) \subset X$ of $x$ such that $\go_{| U(x)}$ is $\ddbar$-exact.  Up to shrinking $U(x)$, we can assume that $U(x) \cap U(x') = \emptyset$ whenever $x \ne x' \in f^{-1}(y)$. Since $f$ is finite, $f$ is open, so $U(y) \cnec \bigcap_{x \in f^{-1}(y)} f(U(x))$ is a neighborhood of $y \in Y$.  We form an open cover $\{U_i\}_{i \in I}$ of $Y$ among these neighborhoods.  For each $i \in I$, there exists by construction a smooth strictly psh function $\phi_i \colon f^{-1}(U_i) \to \bR$ such that $\go_{|f^{-1}(U_i)} = \sqrt{-1} \ddbar\phi_i$.  The functions $\psi_i \colon U \to \bR$ defined by $\psi_i(y) \cnec \sum_{x \in f^{-1}(y)} \phi_i(x)$ (here we regard $f^{-1}(y)$ as a multiset prescribed by its scheme structure) are strictly psh, see~\cite[Lemma II.3.1.2]{VarouchasKahsp}, and the functions $\psi_{ij} = {\psi_i}_{|U_i \cap U_j} - {\psi_j}_{|U_i \cap U_j}$ are pluriharmonic.  The pushforward $f_*[\go] \in H^2(Y,\bR)$ is thus represented by the image of the \v{C}ech 1-cocycle $\psi_{ij}$ under the connecting morphism induced by~\eqref{SE-PH}.  Again by~\cite[Theorem 1]{VarouchasKahsp}, there exist $\cC^\infty$-strictly psh functions $\psi'_i$ defined over $U_i$ such that $\psi'_i - \psi'_j = \psi_{ij}$, so the $\sqrt{-1} \ddbar\psi'_i$ glue to a K\"ahler form on $Y$ which represents $f_*[\go]$.
\end{proof}

\ssec{Gysin morphisms and projection formula for varieties with quotient singularities}

Let $f \colon X \to Y$ be a proper continuous map between two closed rational homology manifolds\footnote{A closed rational homology manifold of dimension $n$ is a compact topological space $X$ such that for every $x \in X$, we have $H_i(X,X \bss \{x\},\bQ) \simeq \bQ$ if $i = n$ and $H_i(X,X \bss \{x\},\bQ) = 0$ if $i \ne n$.} (\eg, complex varieties with at worst quotient singularities; \cf~\cite[Proposition A.1(iii)]{BrionRatsm}).  Then the Poincar\'e duality holds for $X$ and $Y$ (see the proof of~\cite[Section~V.3, Poincar\'e duality~3.2]{IversenBook}), which allows us to define the Gysin morphism
$$f_* \colon H^k(X,\bZ) \xto{\,\PD} H^{\BM}_{\dim X-k}(X,\bZ) \xto{\;f_*\,} H^{\BM}_{\dim X-k}(Y,\bZ) \xto{\,\PD} H^{k-r}(Y,\bZ),$$
where $r = \dim X - \dim Y$ and $\PD$ denotes the Poincar\'e duality between the cohomology groups and the Borel--Moore homology groups $H^{\BM}$.  The following is the reformulation of the projection formula (see \eg~\cite[IX.3.7]{IversenBook}) in terms of Gysin morphism.

\begin{pro}[Projection formula]\label{pro-formproj}
Given $\ga \in H^k(X,\bQ)$ and $\gb \in H^l(Y,\bQ)$, we have
$$f_*(\ga \cdot f^*\gb) = f_*\ga \cdot \gb.$$
\end{pro}

\ssec{An isomorphism statement about Gysin morphisms}

\begin{lem}\label{lem-isomGysin}
Let $f \colon X \to U$ be a surjective proper morphism between complex manifolds with equidimensional connected fibers.  Let $n = \dim X$ and $m =\dim U$.  Assume that $U$ is affine. Then the Gysin morphism
$$f_* \colon H^{2n - m}(X,\bC) \lra H^{m}(U,\bC)$$
is an isomorphism. 
\end{lem}

\begin{proof}
Let $d \cnec n - m$.  Let $E_r^{p,q}$ be the Leray spectral sequence computing $H^\bullet(X,\bC)$ through $f$.  Since every fiber of $f$ has dimension $d$, we have $R^qf_*\bC = 0$ for every $q > 2d$.  As $U$ is affine, by Artin vanishing, see~\cite[Theorem 4.1.26]{Dimca2004}, $E_2^{p,q} \simeq H^p(U,R^qf_*\bC) = 0$ for every $p > m$.  This implies that the only non-vanishing $E_2^{p,q}$ with $p + q = 2n - m$ is $H^m(U,R^{2d}f_*\bC)$, and $E_2^{m,2d} = E_3^{m,2d} = \cdots = E_\infty^{m,2d}$.  So
$$H^{2n-m}(X,\bC) \simeq H^m(U,R^{2d}f_*\bC).$$
	
Let $j \colon U' \hto U$ be the inclusion of a non-empty Zariski open over which $f$ is smooth.  Since $f$ has connected fibers of dimension $d$, $(R^{2d}f_*\bC)_{|U'}$ is isomorphic to the constant sheaf $\bC$ over $U'$ by Poincar\'e duality. So $j_*j^*(R^{2d}f_*\bC) \simeq \bC$, and the natural morphism
$$\Phi \colon R^{2d}f_*\bC \lra j_*j^*(R^{2d}f_*\bC) \simeq \bC$$
has kernel $K \cnec \ker \Phi$ supported on $U \bss U'$.  As $f$ has equidimensional fibers, $f$ has local multi-sections around every point of $U$, which implies that $\Phi$ is surjective.  Since
$$H^m(U,R^{2d}f_*\bC) \simeq H^{2n-m}(X,\bC) \xto{\;f_*\,} H^m(U,\bC)$$
is isomorphic to the morphism induced by $\Phi$, it suffices to prove that $H^i(U,K) = 0$ for $i = m$ and $i = m+1$.  Since $\dim \supp K \le \dim U - 1$, by~\cite[Proposition 5.1.16]{Dimca2004} we have $K[\dim U - 1] \in {}^p D^{\le 0}(U)$, where $p$ is the middle perversity.  It follows from Artin vanishing, see~\cite[Corollary 5.2.18]{Dimca2004}, that $H^i(U,K) = 0$ for every $i \ge \dim U = m$.
\end{proof}

\begin{rem}
In Lemma~\ref{lem-isomGysin}, we need to assume that $f$ is equidimensional.  As a counterexample, let $C$ be any smooth projective curve, and consider the blow-up $\wt{C \times \bC^2}$ of $C \times \bC^2$ at one point.  Then $\wt{C \times \bC^2} \to \bC^2$ does not satisfy the conclusion of Lemma~\ref{lem-isomGysin}.
\end{rem}

\ssec{Maps between cohomology spaces and Hodge classes}

We collect some well-known results about maps between cohomology spaces and include the proofs for completeness.

\begin{lem}\label{lem-fpods}
Let $f\colon X \to Y$ be a surjective morphism between compact complex varieties in the Fujiki class $\cC$.  Assume that $X$ is smooth.  For every integer $k$, we have
$$\ker\(f^* \colon H^k(Y,\bQ) \lra H^k(X,\bQ)\) = \bW_{k-1}H^k(Y,\bQ)$$
and, dually,
$$\Ima\(f_* \colon H_k(X,\bQ) \lra H_k(Y,\bQ)\) = \bW_{-k}H_k(Y,\bQ).$$
\end{lem}

\begin{proof}
Since $H^k(X,\bQ)$ is a pure Hodge structure of weight $k$, the first statement follows from~\cite[Corollary 5.43]{PetersSteenbrinkMHS}.\footnote{The existence of K\"ahler desingularizations of varieties in the Fujiki class $\cC$ allows one to extend Deligne's mixed Hodge theory on complex algebraic varieties to Zariski open subvarieties of compact complex varieties in the Fujiki class $\cC$; \cf~\cite[Section~1]{FujikiDual}. The cited result in~\cite[Chapter 5]{PetersSteenbrinkMHS} proven for complex algebraic varieties generalizes to this larger context.}  Taking the dual, we have
$$\coker\(f_* \colon H_k(X,\bQ) \lra H_k(Y,\bQ)\) = H_k(Y,\bQ)/\bW_{-k}H_k(Y,\bQ),$$
which proves the second statement.
\end{proof}

\begin{lem}\label{lem-excision}
Let $X$ be a compact K\"ahler manifold of dimension $n$, and let $Y \subset X$ be an irreducible closed subvariety of\, $X$.  Let $\ti{\imath} \colon \ti{Y} \xto{\nu} Y \hto X$ be the composition of a K\"ahler desingularization of\, $Y$ with the inclusion $\imath \colon Y \hto X$.  For every $k$, we have
	$$\Ima \(\ti{\imath}_* \colon H_{2n-k}(\ti{Y},\bQ) \lra H^{k}(X,\bQ)\) = \ker\(H^{k}(X,\bQ) \lra H^{k}(X \bss Y,\bQ)\).$$
\end{lem}

\begin{proof}
We have the exact sequence
	\begin{equation}\label{s-excision1}
		H^{k}_Y(X,\bQ) \xto{\;\imath_*\,} H^{k}(X,\bQ) \lra H^{k}(X \bss Y,\bQ), 
	\end{equation}
and by Poincar\'e duality~\cite[Proposition II.9.2]{IversenBook}, $\imath_*$ is isomorphic to
	\begin{equation}\label{s-excision2}
		\imath_* \colon H_{2n-k}(Y,\bQ) \lra  H_{2n-k}(X,\bQ).
	\end{equation}
Since $H_{2n-k}(X,\bQ)$ is a pure Hodge structure of weight $k-2n$, we have
	\begin{equation}\label{s-excision3}
		\imath_* \bW_{k-2n} H_{2n-k}(Y,\bQ) = \Ima(\imath_*)
	\end{equation}
by the strictness of the morphism $\imath_*$ of mixed Hodge structures.  Finally, we have
	\begin{equation}\label{s-excision4}
		\Ima\(\nu_* \colon H_{2n-k}(\ti{Y},\bQ) \lra H_{2n-k}(Y,\bQ) \) = \bW_{k-2n} H_{2n-k}(Y,\bQ)
	\end{equation}
by Lemma~\ref{lem-fpods}.  Combining~\eqref{s-excision1},~\eqref{s-excision2},~\eqref{s-excision3},~\eqref{s-excision4} proves Lemma~\ref{lem-excision}.
\end{proof}

\begin{lem}\label{lem-clHdgctr}
Let $X$ be a compact complex variety in the Fujiki class $\cC$,  and let $\nu \colon \ti{X} \to X$ be a desingularization of\, $X$.  If\, $X$ has at worst rational singularities, then $\ker(\nu_* \colon H_2(\ti{X},\bC) \to H_2(X,\bC))$ consists of Hodge classes.
\end{lem}

\begin{proof}
It suffices to show that, dually,
	$$\coker(\nu^* \colon H^2(X,\bC) \lra H^2(\ti{X},\bC))$$	
consists of Hodge classes.  Since $X$ has at worst rational singularities, the image of $\nu^*$ contains $H^{2,0}(\ti{X})$ by~\cite[Lemma 2.1]{BakkerLehnSingTor}.  Hence the above cokernel consists of Hodge classes.
\end{proof}

\ssec{Campana's criterion}\label{ssec-critCampana}

A compact complex variety $X$ is called \textit{algebraically connected} if for every general pair of points $x,y \in X$, there exists a connected proper curve $C \subset X$ such that $x,y \in C$. The following criterion for a variety in the Fujiki class $\cC$ to be Moishezon is due to Campana.

\begin{thm}[Campana {\cite[Section~3, Corollaire du Th\'eor\`eme~6']{CampanaCored}}]\label{thm-critcamp}
A compact complex variety $X$ in the Fujiki class $\cC$ is Moishezon if and only if it is algebraically connected.
\end{thm}

We list some direct consequences of Campana's criterion.

\begin{cor}\label{cor-critcamp}
Let $X$ be a compact complex variety in the Fujiki class $\cC$, and let $f \colon X \to B$ be a fibration. Assume that a general fiber of $f$\! and $B$ are both Moishezon. Then $X$ is Moishezon if and only if $f$\! has a multi-section.
\end{cor}

\begin{cor}[\cf~{\cite[Corollary 2.12]{HYLkod3}}]\label{cor-critcampP1}
Let $X$ be a compact complex variety in the Fujiki class $\cC$ and $f \colon X \to B$ a $\bP^1$-fibration. If\, $B$ is Moishezon, then $X$ is Moishezon.
\end{cor}

\begin{lem}\label{lem-redalgph}
Let $X$ be a compact K\"ahler manifold with $a(X) = \dim X - 1$.  Then the algebraic reduction $f \colon X \dto B$ of\, $X$ is almost holomorphic, and its general fiber is an elliptic curve.
\end{lem}

Here we recall that a map $f \colon X \dto B$ is called \textit{almost holomorphic} if there exists a non-empty Zariski open $U \subset X$ such that $f_{|U}$ is well defined and proper onto its image.

\begin{proof}
We already know by~\cite[Theorem 12.4]{UenoClassAlgVar} that any resolution of $f$ is an elliptic fibration. It remains to show that $f$ is almost holomorphic. Let $\ti{f} \colon \ti{X} \to B$ be a resolution of $f$ by a compact K\"ahler manifold $\ti{X}$. Let $E \subset \ti{X}$ be the exceptional divisor of $\ti{X} \to X$. Since $X$ is not Moishezon and since $B$ is Moishezon and a general fiber of $\ti{f}$ is a curve, Corollary~\ref{cor-critcamp} implies that $E$ does not dominate $B$. Therefore, $f$ is almost holomorphic.
\end{proof}

\ssec{Bimeromorphic models of compact K\"ahler threefolds with \texorpdfstring{$\boldsymbol{a \le 1}$}{a at most 1}}

Bimeromorphic models of compact K\"ahler threefolds of algebraic dimension $a \le 1$ have essentially been classified by Fujiki~\cite{Fujiki1983} (see also~\cite[Proposition 1.9]{HYLkod3}).  In this subsection, we state this classification result in the following form for our needs.

\begin{pro}\label{pro-class}
Let $X_0$ be a compact K\"ahler threefold such that $a(X_0) \le 1$.  Assume that $X_0$ is not a simple non-Kummer threefold.  Then $X_0$ is bimeromorphic to a compact complex manifold $X$ in the Fujiki class $\cC$, satisfying one of the following descriptions: 
\begin{enumerate}
\item\label{pc-1}  $X$ is the total space of
a $\bP^1$-fibration $X \to S$ over a smooth compact K\"ahler surface $S$.
\item\label{pc-2} $X = (S \times F) / G$, where $S$ is a non-algebraic smooth K\"ahler surface, $F$ is a smooth curve, and $G$ is a finite group acting diagonally on $S \times F$.
\item\label{pc-3} $X$ is the total space of a fibration $f \colon X \to B$ over a smooth projective curve $B$ whose general fiber is an abelian variety.
\item\label{pc-4} $X = \ti{X}/G$, where $G$ is a finite group and $\ti{X}$ is the total space of a $G$-equivariant smooth isotrivial fibration $\ti{f} \colon \ti{X} \to \ti{B}$ in non-algebraic 2-tori without multi-sections.
\item\label{pc-5} $X$ is the quotient $T/G$ of a 3-torus by a finite group.
\end{enumerate}
In~\eqref{pc-1} and~\eqref{pc-3}, we can choose $X$ to be a compact K\"ahler manifold.
\end{pro}

\begin{proof}
First we assume that $X_0$ is uniruled.  Since $X_0$ is non-algebraic, the base of the maximally rationally connected  fibration $X_0 \dto S_0$ is a surface (see \eg~\cite[Proof of Theorem 9.1]{CHPabun}).  Resolving $X_0 \dto S_0$ by some K\"ahler desingularizations of $X_0$ and $S_0$ gives rise to a $\bP^1$-fibration $X \to S$ as in~\eqref{pc-1}.	
Now assume that $X_0$ is not uniruled.  If $a(X_0) = 0$, then by~\cite[Section~1.3, Theorem]{Fujiki1983}, since $X_0$ is not a simple non-Kummer threefold by assumption, necessarily $X_0$ is bimeromorphic to a quotient $T/G$ of a 3-torus $T$ by a finite group $G$ as in~\eqref{pc-5}.
	
Assume that $a(X_0) = 1$; then by~\cite[Section~1.3, Theorem and Section~11.2, Theorem 3]{Fujiki1983}, $X_0$ is bimeromorphic to a threefold $X$ such that
\begin{itemize}
\item  either $X$ satisfies~\eqref{pc-2}; 
\item or $X$ is the total space of a fibration $f \colon X \to B$ over a smooth projective curve such that a general fiber $F$ of $f$ is a 2-torus and $f$ satisfies ``Property (A)'' (namely, for any fibration $g\colon X' \to B$ and any bimeromorphic map $\phi \colon X \dto X'$ over $B$, there exists a Zariski open $U \subset B$ such that $\phi$ induces an isomorphism $f^{-1}(U) \simeq g^{-1}(U)$).
\end{itemize}
	
Consider a fibration $f \colon X \to B$ as in the second case.  By~\cite[Lemma 11.1]{Fujiki1983}, the fact that $f$ has ``Property (A)'' implies that $f$ has no multi-section.  It also implies that if $X' \to X$ is a K\"ahler desingularization of $X$, then a general fiber of the composition $X' \to X \to B$ is still a 2-torus.  Therefore, up to replacing $X$ with $X'$, we can assume that $X$ is a compact K\"ahler manifold.  Assume that $F$ is not algebraic. Then by~\cite[Remark 13.1]{Fujiki1983}, $X_0$ is bimeromorphic to the quotient of the total space of a $G$-equivariant smooth isotrivial torus fibration $\ti{f} \colon \ti{X} \to \ti{B}$ by $G$.  This shows that $f$ satisfies description~\eqref{pc-4} in Proposition~\ref{pro-class}.  Hence the list of bimeromorphic descriptions of $X$ in Proposition~\ref{pro-class} is exhaustive.
\end{proof}

\section{Dual positive cones under dominant meromorphic maps}\label{sec-psefD}

\ssec{Dual K\"ahler cones}

First we prove that the existence of rational classes in the interior of the dual K\"ahler cone is invariant under bimeromorphic modifications (see Proposition~\ref{pro-domduK} for a more general statement).  The statement can be reduced to the special case of blow-ups along smooth centers.

\begin{lem}\label{lem-pullbackKD} 
Let $X$ be a compact K\"ahler manifold, and let $\nu \colon \ti{X} \to X$ be the blow-up of\, $X$ along a submanifold $Y \subset X$. If\, $\Int(\cK(X)^\vee) \cap H^{2n-2}(X,\bQ) \ne \emptyset$, then $\Int(\cK(\ti{X})^\vee) \cap H^{2n-2}(\ti{X},\bQ) \ne \emptyset$.
\end{lem}

In dimension 3, Lemma~\ref{lem-pullbackKD} was proven by Oguiso and Peternell in~\cite[Proposition 2.1]{OguisoPeternellconeKahdual}.  We prove Lemma~\ref{lem-pullbackKD} in arbitrary dimension with a different argument, relying on~\cite[Th\'eor\`eme 1]{Paunthese} as a key ingredient.

\begin{proof}[Proof of Lemma~\ref{lem-pullbackKD}]
Let $E = \nu^{-1}(Y)$ be the exceptional divisor, and let $\ell$ be a line in $\nu^{-1}(y)$ for some $y \in Y$.  Then every element $\ti{\gamma} \in H^{1,1}(\ti{X},\bR)$ is of the form $\ti{\gamma} = \nu^*\gamma - rE$, where $\gamma \cnec \nu_*\ti{\gamma} \in H^{1,1}(X,\bR)$ and $r$ is some real number.  Fix $\ga \in \Int(\cK(X)^\vee) \cap H^{2n-2}(X,\bQ)$.  We will construct a $q \in \bQ_{>0}$ satisfying
	\begin{equation}\label{ineq}
	(\nu^*\ga + q\ell)(\nu^*\gamma - rE) > 0
	\end{equation}
for all $\gamma \in H^{1,1}(X,\bR)$ and $r \in \bR$ such that $\nu^*\gamma - rE \ne 0$ is nef; this implies
	$$\nu^*\ga + q\ell \in \Int(\cK(\ti{X})^\vee) \cap H^{2n-2}(\ti{X},\bQ).$$
	
Note that if $\nu^*\gamma - rE$ is nef, then
$$r = (\nu^*\gamma - rE) \cdot \ell \ge 0.$$
Also note that if moreover $\nu^*\gamma - rE \ne 0$, then $\gamma \ne 0$ because otherwise $-rE \in \ol{\cK(\ti{X})}$ implies $r = 0$.  For every $\gamma \in \nu_*\ol{\cK(\ti{X})} \subset H^{1,1}(X,\bR)$, define
	$$r_\gamma \cnec \min \Set{ r \in \bR | \nu^*\gamma - rE \text{ is nef }} \ge 0.$$
	
Let
$$\cC \cnec \Set{\gamma \in \nu_*\ol{\cK(\ti{X})} | \ga \cdot \gamma > 0}.$$
Fix a norm $\|\cdot\|$ on $H^{1,1}(X,\bR)$.  For every subset $\gS \subset H^{1,1}(X,\bR)$, define
	$$\gS_1 \cnec \Set{\gs \in \gS | \| \gs \| = 1}.$$
	
\begin{claim} We have
		$$R \cnec \inf \Set{r_\gamma | \gamma \in  \(\nu_*\ol{\cK(\ti{X})}\)_1 \bss \cC } > 0.$$
\end{claim}

\begin{proof}
Since $r_\gamma \ge 0$ for each $\gamma$, we have $R \ge 0$.  Assume to the contrary that $R = 0$.  Then there exists a sequence $\gamma_i$ in $ \(\nu_*\ol{\cK(\ti{X})}\)_1 \bss \cC$ such that $\lim_{i \to \infty} r_{\gamma_i} = 0$.  Up to extracting a subsequence, we can assume that  $\gamma \cnec \lim_{i \to \infty} \gamma_i \in H^{1,1}(X,\bR)_1$ exists, so
$$\nu^*\gamma = \lim_{i \to \infty}( \nu^*\gamma_i - r_{\gamma_i} E)$$
is nef.  Thus $\gamma$ is nef by~\cite[Th\'eor\`eme 1]{Paunthese}, so $\ga \cdot \gamma > 0$ (because $\ga \in \Int(\cK(X)^\vee)$).  It follows that $\ga \cdot \gamma_i > 0$ for $i \gg 0$, contradicting $\gamma_i \not\in \cC$.  Hence $R > 0$.
\end{proof}

Let $M < 0$ be such that $M \le \ga \cdot c$ for all $c \in H^{1,1}(X,\bR)_1$.  By the above claim, we can find a $q \in \bQ_{>0}$ such that
$$M + qR > 0.$$
Let $\gamma \in H^{1,1}(X,\bR)$ and $r \in \bR$ be such that $\nu^*\gamma - rE \ne 0$ is nef (so $\gamma \in \nu_*\ol{\cK(\ti{X})} $, $\gamma \ne 0$, and $r \ge 0$).  If $\gamma \in \cC$, then
$$(\nu^*\ga + q\ell)(\nu^*\gamma - rE) = \ga \cdot \gamma + qr > 0.$$	
Suppose that $\gamma \not\in \cC$; then
\begin{equation*}
	\begin{split}
	(\nu^*\ga + q\ell)(\nu^*\gamma - rE) 
	& = \|\gamma\| \(\ga \cdot \frac{\gamma}{\|\gamma\|} + q\frac{r}{\|\gamma\|}\) \\
	& \ge \|\gamma\| \(M + q r_{\gamma/\|\gamma\|}\) \ge 
	\|\gamma\| \(M + qR\) > 0,
	\end{split}
	\end{equation*}
where the first inequality follows from the nefness of $\nu^*\frac{\gamma}{\|\gamma\|} - \frac{r}{\|\gamma\|}E$ and the second inequality from $\frac{\gamma}{\|\gamma\|} \in \(\nu_*\ol{\cK(\ti{X})}\)_1 \bss \cC$.  Hence~\eqref{ineq} holds regardless of whether or not $\gamma \in \cC$, which finishes the proof.
\end{proof}

\begin{pro}\label{pro-domduK} 
Let $f \colon X \dto Y$ be a dominant meromorphic map between compact K\"ahler manifolds.  If\, $X$ satisfies the dual Kodaira condition (K), then so does $Y$.
\end{pro}

\begin{proof}
Let $X \xlto{\nu} \ti{X} \xto{\ti{f}} Y$ be a resolution of $f$ by a sequence of blow-ups $\nu \colon \ti{X} \to X$ along smooth centers.  By Lemma~\ref{lem-pullbackKD}, there exists an $\ti{\ga} \in \Int(\cK(\ti{X})^\vee) \cap H_2(\ti{X},\bQ)$.  We conclude by~\cite[Proposition 2.5]{OguisoPeternellconeKahdual} that $\ti{f}_*\ti{\ga} \in \Int(\cK(Y)^\vee) \cap H_2(Y,\bQ)$.
\end{proof}

We also have the following for dual K\"ahler cones.

\begin{lem}\label{lem-finietintK}
Let $f \colon X \to Y$ be a finite morphism between compact K\"ahler manifolds.  If $\ga \in \Int(\cK(Y)^\vee)$, then $f^*\ga \in \Int(\cK(X)^\vee)$.
\end{lem}

\begin{proof}
Fix a K\"ahler class $\go_Y \in H^2(Y,\bR)$ on $Y$.  Since $f$ is finite, $f^*\go_Y$ is a K\"ahler class on $X$.  Thus for every non-zero nef class $\go \in \ol{\cK(X)}$, we have $(f_*\go) \cdot \go_Y^{n-1} = \go \cdot f^*\go_Y^{n-1} > 0$, where $n = \dim X$, so $f_*\go \ne 0$ in $H^2(Y,\bR)$.  By Lemma~\ref{lem-pousKah}, we have $f_*\go \in \ol{\cK(Y)}$.  It follows that $f^*\ga \cdot \go = \ga \cdot f_*\go > 0$.  Hence $f^*\ga \in \Int(\cK(X)^\vee)$.
\end{proof}

\ssec{Dual pseudoeffective cones}

We start with the easy observation that the interior of the dual pseudoeffective cone is stable under pushforwards by surjective morphisms.

\begin{lem}\label{lem-push}
Let $f \colon X \to Y$ be a surjective map between compact K\"ahler manifolds. If $\ga \in \Int(\Psef(X)^\vee)$, then $f_*\ga \in \Int(\Psef(Y)^\vee)$.
\end{lem}

\begin{proof}
For every $\gamma \in \Psef(Y) \bss\{0\}$, since $\ga \in \Int(\Psef(X)^\vee)$ and $f^*\gamma \in \Psef(X) \bss\{0\}$, we have $(f_*\ga) \cdot \gamma = \ga \cdot f^*\gamma > 0$. Hence $f_*\ga \in \Int(\Psef(Y)^\vee)$.
\end{proof}

The following result is the analog of Lemma~\ref{lem-pullbackKD} for pseudoeffective cone.

\begin{pro}\label{pro-pullback}
Let $X$ be a compact K\"ahler manifold, and let $\nu \colon \ti{X} \to X$ be the blow-up of $X$ along a submanifold $Y \subset X$. If\, $\Int(\Psef(X)^\vee) \cap H_2(X,\bQ) \ne \emptyset$, then $\Int(\Psef(\ti{X})^\vee) \cap H_2(\ti{X},\bQ) \ne \emptyset$.
\end{pro}

\begin{proof}
Fix $\ga \in \Int(\Psef(X)^\vee) \cap H^{2n-2}(X,\bQ)$.  Let $E = \nu^{-1}(Y)$ be the exceptional divisor, and let $\ell$ be a line in $\nu^{-1}(y)$ for some $y \in Y$. Then every element $\ti{\gamma} \in H^{1,1}(\ti{X},\bR)$ is of the form $\ti{\gamma} = \nu^*\gamma + r E$ for some $\gamma \in H^{1,1}(X,\bR)$ and $r \in \bR$.  If moreover $\ti{\gamma} \in \Psef(\ti{X})$, then $\gamma = \nu_*\ti{\gamma}$ is also pseudoeffective.  Therefore to prove the proposition, it suffices to find a $q \in \bQ_{>0}$ such that $(\nu^*\ga - q \ell) \cdot (\nu^*\gamma + r E) > 0$ for every $\gamma \in \Psef(X)$ and $r \in \bR$ such that $\nu^*\gamma + r E \in \Psef(\ti{X}) \bss \{0\}$.

Fix a norm $\|\cdot\|$ on $H^2(X,\bR)$, and let
$$\Psef(X)_1 = \Psef(X) \cap \left\{\gamma \in H^2(X,\bR) \mid \| \gamma \| = 1 \right\}.$$ 
For every $\gamma \in \Psef(X)_1$, let $r_\gamma = \inf \left\{r \in \bR \mid \nu^*\gamma + r E \in \Psef(\ti{X}) \right\}$. Since $\nu^*\gamma \in \Psef(\ti{X})$, we have $r_{\gamma} \le 0$. As $\ga \cdot \gamma > 0$ for every $\gamma \in \Psef(X)_1$ and both $\gamma \mapsto \ga \cdot \gamma$ and $\gamma \mapsto r_\gamma$ are continuous functions defined on the compact set $\Psef(X)_1$, there exists a $q \in \bQ_{>0}$ such that
$$\ga \cdot \gamma + q r_\gamma >0$$ 
for all $\gamma \in \Psef(X)_1$.

Now let $\gamma \in \Psef(X)$ and $r \in \bR$ be such that $\nu^*\gamma + r E \in \Psef(\ti{X}) \bss \{0\}$. If $\gamma = 0$, then $r > 0$, so
$$(\nu^*\ga - q \ell) \cdot (r  E) = qr > 0.$$
If $\gamma \ne 0$, then $\frac{r}{\|\gamma \| } \ge r_{\gamma / \|\gamma \|}$, so we also have
	\begin{equation*}\pushQED{\qed}
(\nu^*\ga - q \ell) \cdot (\nu^*\gamma + r  E) = \ga \cdot \gamma + qr = \|\gamma \| \(\ga \cdot \frac{\gamma}{\|\gamma \| }+ q  \frac{r}{\|\gamma \| } \) \ge \|\gamma \| \(\ga \cdot \frac{\gamma}{\|\gamma \| }+ q  r_{\gamma / \|\gamma \|} \)  > 0.\qedhere \popQED
	\end{equation*}
\renewcommand{\qed}{}     
\end{proof}

\begin{rem}
In the setting of Proposition~\ref{pro-pullback}, it is not true that $\ga \in \Int\(\Psef(X)^\vee\)$ implies $\nu^*\ga \in \Int(\Psef(\ti{X})^\vee)$ (which is already false when $\nu$ is the blow-up of $\bP^2$ along a point).  However, note that for any generically finite surjective morphism $f \colon X \to Y$ between compact K\"ahler manifolds, since $f_*\Psef(X) \subset \Psef(Y)$, we always have $f^*\Psef(Y)^\vee \subset \Psef(X)^\vee$.
\end{rem}

As an immediate consequence of Lemma~\ref{lem-push} and Proposition~\ref{pro-pullback}, we have the following.

\begin{cor}\label{cor-dompseudu}
Let $f \colon X \dto Y$ be a dominant meromorphic map between compact K\"ahler manifolds.  If\, $X$ satisfies the dual Kodaira condition (P), then so does $Y$.
\end{cor}

\begin{proof}
Let $X \xlto{\nu} \ti{X} \xto{\ti{f}} Y$ be a resolution of $f$ by a sequence of blow-ups $\nu \colon \ti{X} \to X$ along smooth centers. By Proposition~\ref{pro-pullback}, there exists an $\ti{\ga} \in \Int(\Psef(\ti{X})^\vee) \cap H_2(\ti{X},\bQ)$.  We conclude by Lemma~\ref{lem-push} that $\ti{f}_*\ti{\ga} \in \Int(\Psef(Y)^\vee) \cap H_2(Y,\bQ)$.
\end{proof}

\section{Smooth torus fibrations}\label{sec-fibtlis}

In this section, we study the Oguiso--Peternell problem for smooth torus fibrations. The argument involves the Deligne cohomology in an essential way, and the reader is referred to, \eg, \cite[Sections~2 and~3]{EZnormfct} for a reference.  See also~\cite[Section~2]{ClaudonToridefequiv}.

Let $f \colon X \to B$ be a smooth torus fibration such that $X$ and $B$ are compact K\"ahler manifolds, and let $g = \dim X - \dim B$.  Recall that the (absolute) Deligne complex $\ul{D}_X(g)$ is defined as
$$\cdots 0 \lra \bZ_X \xto{\times (2\pi \sqrt{-1})^g} \cO_X \xto{\;\;d\;\;} \gO_X^1 \xto{\;\;d\;\;} \cdots \xto{\;\;d\;\;} \gO_X^{g-1} \lra 0 \lra \cdots,$$
where $\bZ_X$ is placed at the $\supth{0}$ degree. 
The Deligne cohomology group of degree $2g$ is defined by 
$H^{2g}_D(X,\bZ(g)) = \bH^{2g}(X,\ul{D}_{X}(g))$.
We have the short exact sequence of complexes
\begin{equation}\label{cplx-DX}
0  \lra \gO_X^{\bullet \le g-1}[-1]  \lra \ul{D}_X(g) \lra \bZ_X \lra 0
\end{equation}
which gives rise to the regulator map 
\begin{equation}\label{map-reg}
\cl \colon H^{2g}_D(X,\bZ(g)) = \bH^{2g}(X,\ul{D}_{X}(g)) \lra H^{2g}(X,\bZ)
\end{equation}
with image equal to 
$$\ker\(H^{2g}(X,\bZ) \lra \bH^{2g}(X,\gO_X^{\bullet \le g-1})\) = H^{g,g}(X,\bZ).$$ 

Similarly, the relative Deligne complex $\ul{D}_{X/B}(g)$ is defined as
$$\cdots 0 \lra \bZ_X \xto{\times (2\pi \sqrt{-1})^g} \cO_{X/B} \xto{\;\;d\;\;} \gO_{X/B}^1 \xto{\;\;d\;\;} \cdots \xto{\;\;d\;\;} \gO_{X/B}^{g-1} \lra 0 \lra \cdots,$$
and we have the short exact sequence of complexes
\begin{equation}\label{cplx-DXB}
0  \lra \gO_{X/B}^{\bullet \le g-1}[-1]  \lra \ul{D}_{X/B}(g) \lra \bZ_X \lra 0.
\end{equation}
Applying $Rf_*$ to~\eqref{cplx-DXB}, together with the vanishing $R^{2g}f_*\gO^{\bullet \le g-1}_{X/B} = 0$ (because $f$ is smooth of relative dimension $g$), we obtain a short exact sequence
\begin{equation}\label{SE-DeligneJac}
0  \lra \cJ  \lra R^{2g}f_*\ul{D}_{X/B}(g) \lra 
R^{2g}f_*\bZ_X \simeq \bZ_B \lra 0, 
\end{equation}
where 
$$\cJ \cnec \coker \(R^{2g-1}f_*\bZ_X \lra R^{2g-1}f_*\gO^{\bullet \le g-1}_{X/B}\).$$

The natural map $\gO_{X}^{\bullet} \to \gO_{X/B}^{\bullet}$ of de Rham complexes induces a morphism of exact sequences from~\eqref{cplx-DX} to~\eqref{cplx-DXB}, which further induces the commutative diagram
 \begin{equation}\label{SE-DeligneJaccomm}
\begin{tikzcd}[cramped]
 H^{2g}_D(X,\bZ(g)) \ar[d] \ar[r, twoheadrightarrow, "\cl"] & H^{g,g}(X,\bZ) \ar[d, "f_*"]   \\
H^0(B, R^{2g}f_*\ul{D}_{X/B}(g)) \arrow[r] & H^0\(B,\bZ\)  \ar[r,"\gd"] & H^1(B, \cJ)\rlap{,}
\end{tikzcd}
\end{equation}
where the vertical arrow on the left is the composition of $H^{2g}_D(X,\bZ(g)) \to \bH^{2g}(X,\ul{D}_{X/B}(g))$ induced by the natural map $\ul{D}_X(g) \to \ul{D}_{X/B}(g)$ with $\bH^{2g}(X,\ul{D}_{X/B}(g)) \to H^0(B, R^{2g}f_*\ul{D}_{X/B}(g))$ and the second row is an exact sequence induced by~\eqref{SE-DeligneJac}.

Finally, the sheaf $\cJ$ is isomorphic to the sheaf of germs of sections of the Jacobian fibration $p \colon J \to B$ associated to $f$.  Set $\eta(f) \cnec \gd(1) \in H^1(B, \cJ)$. This defines a bijection
\begin{equation}\label{corr-Jtors}
\eta \colon \Set{ \text{Isomorphism classes of } J\text{-torsors} }   \elra H^1(B, \cJ),
\end{equation}
and $\eta(f)$ is torsion if and only if $f$ has a multi-section (which can be chosen \'etale over $B$); \cf~\cite[Propositions 2.1 and 2.2]{ClaudonToridefequiv}.

\begin{lem}\label{lem-fibTlis}
Let $f \colon X \to B$ be a smooth torus fibration of relative dimension $g$ over a compact complex manifold.  If $f_* \colon H^{g,g}(X,\bQ) \to H^0(B,\bQ)$ is surjective, then $f$\! has an \'etale multi-section.
\end{lem}
 
\begin{proof}
It suffices to show that $\eta(f)$ is torsion.  Since $f_* \colon H^{g,g}(X,\bQ) \to H^0(B,\bQ)$ is surjective, there exists an $\ga \in H^{g,g}(X,\bZ)$ such that $f_*\ga \in H^0(B,\bZ)$ is non-zero. By~\eqref{SE-DeligneJaccomm}, $f_*\ga$ lifts to an element of $H^0(B, R^{2g}f_*\ul{D}_{X/B}(g))$, so $\gd(f_*\ga) = 0$.  Hence $\eta(f)= \gd(1)$ is torsion.
\end{proof}

\begin{pro}\label{pro-fibTlis}
Let $X_0$ be a compact K\"ahler manifold of dimension $n$ such that $\Int(\cK(X_0)^\vee)$ contains a rational class $\ga \in H^{2n-2}(X_0,\bQ)$.  Assume that $X_0$ is bimeromorphic to $X/G$, where $G$ is a finite group and $X$ is the total space of a $G$-equivariant smooth torus fibration $f \colon X \to B$ over a smooth curve $B$. Then $f$\! has an \'etale multi-section.
\end{pro}

\begin{proof}
Let $X_0 \lto \ti{X} \to X/G$ be a resolution of a bimeromorphic map $X_0 \dto X/G$ by a sequence of blow-ups $\ti{X} \to X_0$ along smooth centers.  By Lemma~\ref{lem-pullbackKD}, $\Int(\cK(\ti{X})^\vee)$ contains a rational class $\gb$.  Let $X \xleftarrow{p} \ti{X}' \xto{q} \ti{X}$ be a resolution of the meromorphic map $X \to X/G \dto \ti{X}$ by a compact K\"ahler manifold~$\ti{X}'$. The situation is summarized in the commutative diagram
$$
\begin{tikzcd}[cramped]
\ti{X}'  \ar[r, "p"] \ar[d, "q"] & X \ar[r, "f"]\ar[d] & B \ar[d,"r"] \\
\ti{X} \ar[r] \ar[rr, bend right=25, swap, "\ti{f}"]  & X/G \ar[r] & B/G\rlap{.}
\end{tikzcd}
$$
With the notation therein,
$$r_*f_*p_*q^*\gb = \ti{f}_*q_*q^*\gb = \deg(q) \cdot \ti{f}_* \gb \ne 0 \in H^0(B/G,\bQ),$$
where the non-vanishing follows from~\cite[Proposition 2.5]{OguisoPeternellconeKahdual}. In particular, if $\ga \cnec p_*q^*\gb $, then $f_* \ga \ne 0$ in $H^0(B,\bQ)$. Therefore, $f_* \colon H^{n-1,n-1}(X,\bQ) \to H^0(B,\bQ)$ is surjective, and we apply Lemma~\ref{lem-fibTlis} to conclude.
\end{proof}

\section{Algebraicity of the Albanese torus}\label{sec-tores}

The aim of this section is to prove Theorem~\ref{thm-alb}, answering Problem~\ref{prob-OP} in terms of the Albanese torus for every compact K\"ahler manifold.

\looseness=-1 First we study the Oguiso--Peternell problem for complex tori (and their finite quotients for later use).

\begin{pro}\label{pro-T}
Let $T$ be a complex torus of dimension $n$ and $G$ a finite group acting on $T$. If there exists a $\gb \in H^{n-1,n-1}(T,\bQ)^G$ such that $\gb \cdot \go \ne 0$ for every $\go \ne 0 \in \ol{\cK(T)}^G$, then $T$ is projective.
\end{pro}

We will need the following lemma. 

\begin{lem}[Poincar\'e's formula]\label{lem-Poincare}
Let $L$ be a line bundle on the complex torus $T$ of dimension $n$.  Assume that
$$d \cnec \frac{c_1(L)^n}{n!} \ne 0,$$ 
and let 
$$c_L \cnec \frac{c_1(L)^{n-1}}{(n-1)! d} \in H^{2n-2}(T,\bQ).$$
Then for every integer $p \in [0,n]$, we have
$$\frac{c_1(L)^p}{p!} = d\frac{c_L^{\st n-p}}{(n-p)!} \in H^{2p}(T,\bQ),$$
where $\st$ is the Pontryagin product on $H^\bullet(T,\bQ)$.
\end{lem}

This lemma is well known; we provide a proof for the sake of completeness.

\begin{proof}
Let $\gl_1,\ldots,\gl_n,\mu_1,\ldots,\mu_n$ be a symplectic basis of $H_1(T,\bZ)$ for $L$ (see~\cite[Section~3.1]{CabV}), and let $dx_1,\ldots,dx_n, dy_1,\ldots, dy_n$ be the associated dual basis of $H^1(T,\bZ)$.  Assume that $L$ is of type $(d_1,\ldots,d_n)$.  By~\cite[Lemma 3.6.4]{CabV}, we have
\begin{equation}\label{eqn-c1L}
	c_1(L) = -\sum_{i= 1}^n d_i \cdot dx_i \wedge dy_i
\end{equation}
and also 
\begin{equation}\label{eqn-d}
	d = (-1)^s d_1\cdots d_n
\end{equation}
by~\cite[Theorem 3.6.1]{CabV}, where $s$ is number of negative eigenvalues of the Hermitian form associated to $L$.  A computation thus shows that
\begin{equation}\label{eqn-cL}
c_L =  \frac{c_1(L)^{n-1}}{(n-1)! d} 
= \sum_{i=1}^n \frac{(-1)^{s+n-1}}{d_i}\wh{dx_i \wedge dy_i},
\end{equation}
where 
$$\wh{dx_i \wedge dy_i} \cnec( dx_1 \wedge dy_1 )\wedge \cdots \wedge (dx_{i-1} \wedge dy_{i-1} )
\wedge (dx_{i+1} \wedge dy_{i+1})  \wedge \cdots \wedge
(dx_n \wedge dy_n).$$

Let $\PD\colon H^\bullet(T,\bZ) \to H_{2n-\bullet}(T,\bZ) $ denote the Poincar\'e duality morphism.  For every subset $I$ of $[1, n] \cap \bZ$, we deduce from~\cite[Lemmas~3.6.5 and~4.10.1]{CabV} that
\begin{equation}\label{eqn-PD}
\PD\(\bigwedge_{i \in  I}(dx_i \wedge dy_i)\) = (-1)^{n+s} 
\underset{{i \in  I^\circ}}{\mathlarger{\mathlarger{\bigstar}}} (\gl_i \star \mu_i), 
\end{equation}
where $I^\circ \cnec [1, n] \cap \bZ \bss I$.  It follows from~\eqref{eqn-cL} and~\eqref{eqn-PD} that
$$ \PD\(\frac{c_L^{\st n-p}}{(n-p)!}\) =\frac{\PD(c_L)^{\st n-p}}{(n-p)!} = (-1)^{n-p} \sum_{I} \(\prod_{i \in I} d_i \)^{-1} \underset{{i \in I}}{\mathlarger{\mathlarger{\bigstar}}} (\gl_i \star \mu_i),$$
where $I $ in the sum runs through all subsets of $[1, n] \cap \bZ$ of cardinal $n-p$, and thus
$$ \frac{c_L^{\st n-p}}{(n-p)!} = (-1)^{s - p} \sum_{I} 
\(\prod_{i \in  I} d_i \)^{-1} \bigwedge_{i \in  I^\circ}(dx_i \wedge dy_i),$$
again by~\eqref{eqn-PD}.
Hence
$$d \cdot  \frac{c_L^{\st n-p}}{(n-p)!} = (-1)^{p} \sum_{I} 
\(\prod_{i \in  I^\circ} d_i \)  \bigwedge_{i \in  I^\circ}(dx_i \wedge dy_i)
= \frac{c_1(L)^p}{p!}$$
by~\eqref{eqn-d} and~\eqref{eqn-c1L}.
\end{proof}

\begin{proof}[Proof of Proposition~\ref{pro-T}]	
We can assume that the group action is trivial.  Indeed, if $\gb$ satisfies the assumption of Proposition~\ref{pro-T}, then the same $\gb$ satisfies
$$\gb \cdot \go = 
\frac{1}{|G|}\sum_{g \in G} g_*\gb \cdot \go 
=\gb \cdot \frac{1}{|G|}\sum_{g \in G} g^*\go \ne 0$$
for every $\go \ne 0 \in \ol{\cK(T)}$.  Up to replacing $\gb$ with some non-zero multiple of it, we can also assume that $\gb$ is the image of some element of $H^{n-1,n-1}(T,\bZ)$ (still denoted by $\gb$).

We first prove Proposition~\ref{pro-T} assuming $T$ is a simple torus (in the sense that $T$ does not contain any non-trivial sub-torus).  Let $n = \dim T$, and let $\hat{T}$ be the dual of $T$. Let
$$\cF \colon H^{2n-k}(T,\bZ) \elra H^k(\hat{T},\bZ)$$
be the Fourier transformation.  Up to sign, $\cF$ is the composition of the Poincar\'e duality morphism $H^{2n-k}(T,\bZ) \eto H^k(T,\bZ)^\vee$ with the isomorphism $H^k(T,\bZ)^\vee \eto H^k(\hat{T},\bZ)$ induced by the natural perfect pairing $H^1(T,\bZ) \otimes H^1(\hat{T},\bZ) \to \bZ$, \cf~\cite[Proposition 1]{BeauvilleTransfFM}, and it follows that
$$\cF(\gamma_1 \star \gamma_2) = \pm \cF(\gamma_1) \cdot \cF(\gamma_2) \quad \text{and} \quad \cF(\gamma_1 \cdot \gamma_2) = \pm \cF(\gamma_1) \star \cF(\gamma_2).$$
The map $\cF$ is an isomorphism of Hodge structures, so there exists a line bundle $L$ on $\hat{T}$ such that $c_1(L) = \cF(\gb)$; let $\phi_L \colon \hat{T} \to T$ be the homomorphism induced by $L$. Since $T$ is simple, $\ker(\phi_L)$ is either $\hat{T}$ or finite. In other words, either $c_1(L) = 0$, or $\phi_L$ is finite; \cf~\cite[Lemma 2.4.7]{CabV}.  As $\gb \ne 0$, we have $c_1(L) = \cF(\gb) \ne 0$.  It follows from~\cite[Corollary 3.6.2 and Theorem 3.6.3]{CabV} that $ \(\frac{c_1(L)^n}{n!}\)^2 = \deg \phi_L \ne 0$.  So $c_1(L)^n \ne 0$, and $\gb^{\st n} = \pm\cF^{-1}(c_1(L)^{n}) \ne 0$.  In particular, $\gb^{\st (n-1)} \ne 0$.

Let $L'$ be a line bundle over $T$ such that $c_1(L') = \gb^{\st(n-1)}$. By the same argument, we have $c_1(L')^n \ne 0$, so the Hermitian form $h$ on $H_1(T,\bC)$ which corresponds to $c_1(L')$ is non-degenerate. Let $(p,q)$ be the signature of $h$. There exist $dz_1,\ldots, dz_n \in H^1(T,\bC)$ such that $dz_1,\ldots, dz_n, d\bar{z}_1,\ldots, d\bar{z}_n$ form a basis of $H^1(T,\bC)$ and
$$c_1(L') =  \sqrt{-1} \sum_{j=1}^n c_j\,  dz_j \wedge d \bar{z}_j \in H^1(T,\bR)$$
with $c_1,\ldots, c_p > 0$ and $c_{p+1},\ldots, c_n < 0$. Define
$$\go \cnec \sqrt{-1} \( q \cdot \sum_{j=1}^p c_j\,  dz_j \wedge d \bar{z}_j -  p \cdot \sum_{j= p +1}^n c_j\,  dz_j \wedge d \bar{z}_j  \),$$ 
which is a K\"ahler form.

Assume that $p , q \ne n$. Then $\go \cdot c_1(L')^{n-1} = 0$ by an elementary computation.  Since $c_1(L') = \gb^{\st(n-1)}$ and $\gb^{\st n} \ne 0$, we have $\hat{\cF}(c_1(L')) = \pm \hat{\cF}(\gb)^{n-1}$ and $\hat{\cF}(\gb)^n \ne 0$, where $\hat{\cF} \colon H^\bullet(\hat{T},\bZ) \to H^\bullet(T,\bZ)$ denotes the Fourier transform of $\hat{T}$.  So by Lemma~\ref{lem-Poincare}, there exists a $C \in \bQ \bss \{0\}$ such that
$$\hat{\cF}(c_1(L')^{n-1}) = \pm \hat{\cF}(c_1(L'))^{\st (n-1)}
= \pm \(\hat{\cF}( \gb)^{n-1}\)^{\st (n-1)}  = C \hat{\cF}(\gb),$$ 
and therefore 
$$c_1(L')^{n-1} = C \gb.$$
It follows that
$$ \go \cdot c_1(L')^{n-1} = C \(\go \cdot \gb \) \ne 0,$$ 
where the non-vanishing follows from the assumption on $\gb$ and $\go \in \cK(T)$.  This contradicts the equality $\go \cdot c_1(L')^{n-1} = 0$.  Hence either $p = n$, or $q = n$, so either $L'$ or $L'^\vee$ is ample. Thus $T$ is projective.

Now we prove Proposition~\ref{pro-T} by induction on $\dim T$.  When $\dim T = 1$, $T$ is always projective. Assume that $\dim T > 1$ and that Proposition~\ref{pro-T} is proven for every complex torus of dimension strictly less than $\dim T$ endowed with a finite group action.  By what we have proven, we can assume that $T$ is not simple.  Then $T$ is the total space of a smooth isotrivial torus fibration $\pi \colon T \to T'$ over a simple complex torus $T'$ with $0 < \dim T' < \dim T$.  For every $\go' \in \ol{\cK(T')} \bss \{0\}$, we have $\pi^*\go'\in \ol{\cK(T)} \bss \{0\}$, so
$$\pi_*\gb \cdot \go' = \gb \cdot \pi^*\go'  \ne 0$$
by assumption.  Thus $T'$ is projective by the induction hypothesis.  Moreover, $\pi_*\gb \ne 0 \in H_2(T',\bZ)$, and since $T'$ is simple, the previous argument showing that $\gb^{\st n} \ne 0$ (when $T$ is simple) also proves that
$$\pi_*(\gb^{\st\dim T'}) = (\pi_*\gb)^{\st\dim T'} \ne 0 \in H^0(T',\bZ),$$
so $\pi$ has an \'etale multi-section $\gS \subset T$ by Lemma~\ref{lem-fibTlis}.  It follows that there exists a finite \'etale cover $\tau \colon \ti{T} \to T$, together with a surjective homomorphism $\pi' \colon \ti{T} \to F$ over a fiber $F \subset T$ of $\pi \colon T \to T'$.  For every $\go'' \ne 0 \in \ol{\cK(F)}$, we have $\tau_*\pi'^*\go'' \ne 0 \in \ol{\cK(T)}$, so
$$\pi'_*\tau^*\gb \cdot \go'' = \gb \cdot \tau_*\pi'^*\go'' \ne 0.$$ 
Thus $F$ is projective as well by the induction hypothesis.  Since $\pi \colon T \to T'$ is a fibration with a multi-section such that $T'$ and the fibers of $\pi$ are projective, by Corollary~\ref{cor-critcamp}, $T$ is also projective.
\end{proof}

Before we prove Theorem~\ref{thm-alb}, let us prove some immediate consequences of Proposition~\ref{pro-T}.

\begin{cor}\label{cor-T/G}
Let $X$ be a compact K\"ahler manifold which is bimeromorphic to the quotient $T/G$ of a complex torus $T$ by a finite group $G$.  If\, $X$ satisfies the dual Kodaira condition (K), then $X$ is projective.
\end{cor}

\begin{proof}
Let $n \cnec \dim X$.  Let $\ga \in \Int\(\cK(X)^\vee\) \cap H^{2n-2}(X,\bQ)$.  Fix a bimeromorphic map $p \colon X \dto T/G$. Up to resolving $p$ by successive blow-ups of $X$ along smooth centers, by Lemma~\ref{lem-pullbackKD} we can assume that $p$ is holomorphic.  Let $q \colon T \to T/G$ be the quotient map. For every $\go \ne 0 \in \ol{\cK(T)}^G$, there exists by Lemma~\ref{lem-pullbackKahG} a nef class $\go' \ne 0 \in H^2(T/G,\bR)$ such that $\go = q^*\go'$.  As $\go'$ is nef, the pullback $p^*\go'$ is also nef.  So if we set $\gb \cnec q^*p_*\ga \in H^{n-1,n-1}(T,\bQ)^G$, then since $\ga \in \Int\(\cK(X)^\vee\) \cap H^{2n-2}(X,\bQ)$, we have
\begin{equation}\label{eqn-pos}
\gb \cdot \go = q^*p_*\ga \cdot \go = q^*(p_*\ga \cdot \go') =  q^*p_*(\ga \cdot p^*\go') > 0,
\end{equation}
where the last equality follows from the projection formula (\cf~Proposition~\ref{pro-formproj}).  It follows from Proposition~\ref{pro-T} that $T$ is projective; hence $X$ is also projective.
\end{proof}

\begin{cor}\label{cor-b1}
Let $X$ be compact K\"ahler manifold which satisfies the dual Kodaira condition (K). If $a(X) = 0$, then $b_1(X) = 0$.
\end{cor}

\begin{proof}
It is equivalent to show that $\Alb(X)$ is a point. Since $a(X) = 0$, the Albanese map $X \to \Alb(X)$ is surjective and $a(\Alb(X)) = 0$; \cf~\cite[Lemma 13.1]{UenoClassAlgVar}. So the dual Kodaira condition (K) implies that $\Int\(\cK(\Alb(X))^\vee\)$ contains a rational class as well by Proposition~\ref{pro-domduK}.  It follows from Proposition~\ref{pro-T} that $\Alb(X)$ is projective; hence $\Alb(X)$ is a point.
\end{proof}

We finish this section with a proof of Theorem~\ref{thm-alb}.

\begin{proof}[Proof of Theorem~\ref{thm-alb}]
Let $\alb \colon X \to T$ be the Albanese map of $X$.

\begin{claim}
For every $[\go] \in \ol{\cK(T)} \bss \{0\}$, we have $\alb^*[\go] \in \ol{\cK(X)} \bss \{0\}$.
\end{claim}

\begin{proof}[Proof of the claim]
Clearly, $\alb^*[\go] \in \ol{\cK(X)}$. It remains to show that $\alb^*[\go] \ne 0$.  As $[\go]$ is nef, we can assume that $[\go]$ is represented by $\go = \sqrt{-1} \sum_{j=1}^n c_j\, dz_j \wedge d \bar{z}_j$ for some complex coordinates $(z_1,\ldots,z_n)$ of $H_1(T,\bC)$ and some $c_j \ge 0$ such that $c_j \ne 0$ for some $j = 1,\ldots,n$.  It follows that $\alb^*\go$ is semi-positive, so it suffices to show that $\alb^*\go \ne 0$.
	
Assume to the contrary that $\alb^*\go = 0$. Let $Y \cnec \alb(X)$, and let $Y^\circ \subset Y$ be a non-empty Zariski open over which $\alb \colon X \to Y$ is smooth.  Then for any point $y \in Y^\circ$ and any vector $v \in T_{Y,y}$, we have $\go(v \wedge \bar{v}) = 0$. Since $Y$ is the Albanese image of $X$, if we regard each $T_{Y,y}$ as a subspace of $T_{T,o}$ by translation, where $o \in T$ is a chosen origin of $T$, the subspaces $T_{Y,y}$ generate $T_{T,o}$, where $y$ runs through $Y^\circ$.  As $\go_{|o}$ is positive semi-definite, we thus have $\go_{|o} = 0$, contradicting the assumption that $\go \ne 0$.  Hence $\alb^*\go \ne 0$.
\end{proof}

Now let $\ga \in \Int\(\cK(X)^\vee\) \cap H^{2n-2}(X,\bQ)$.  By the above claim, we have $[\go] \cdot \alb_*\ga = \alb^*[\go] \cdot \ga > 0$ for every $[\go] \in \ol{\cK(T)} \bss \{0\}$.  We conclude by Proposition~\ref{pro-T} that $T$ is projective.
\end{proof}

\section{Ricci-flat manifolds}\label{sec-Rp}

In this section, we prove Theorem~\ref{thm-RP}, answering Problem~\ref{prob-OP} for Ricci-flat compact K\"ahler manifolds.  Let us start with the special case of hyper-K\"ahler manifolds.

\ssec{Hyper-K\"ahler manifolds}\label{ssec-HK}

In this article, a compact hyper-K\"ahler manifold is a simply connected compact K\"ahler manifold $X$ such that $H^0(X,\gO_X^2)$ is generated by a nowhere degenerate holomorphic $2$-form.  The reader is referred to~\cite{HuybHK} for basic results about compact hyper-K\"ahler manifolds.

Let $X$ be a compact hyper-K\"ahler manifold.  Let $q_X$ be the Beauville--Bogomolov--Fujiki quadratic form on $H^2(X,\bR)$, and let
$$\Phi_X \colon H^2(X,\bR) \simeq H^{2n-2}(X,\bR)$$ 
be the isomorphism (of Hodge structures) sending $\ga \in H^2(X,\bR)$ to the Poincar\'e dual of $q_X(\ga , \cdot) \in H^2(X,\bR)^\vee$.  For every cone $C \subset H^{1,1}(X,\bR)$, its dual with respect to $q_X$ is denoted by
$$C^* \cnec \left\{\ga \in H^{1,1}(X,\bR) \mid q_X(\ga , \gb) \ge 0 \text{ for every } \gb \in C \right\}.$$
We have
$$\Phi_X(C^*) = C^\vee,$$
where $C^\vee \subset H^{n-1,n-1}(X,\bR)$ is the Poincar\'e dual of $C$.

For a compact hyper-K\"ahler manifold $X$, there are two other natural positive cones that we can define in $H^{1,1}(X,\bR)$: the birational K\"ahler cone $\cB\cK (X)$ and the positive cone $\cC(X)$. Recall that
$$\cB\cK (X) \cnec \bigcup_{f \colon X \dto X'} f^*\cK(X') \subset H^{1,1}(X,\bR),$$ 
where the union runs through all bimeromorphic maps $f$ from $X$ to another compact hyper-K\"ahler manifold and $\cC(X)$ is defined to be the connected component of
$$\left\{\ga \in  H^{1,1}(X,\bR) \mid q_X(\ga,\ga) > 0 \right\}$$
containing $\cK(X)$. 

We are able to answer Problem~\ref{prob-OP} for hyper-K\"ahler manifolds based on Huybrechts' description of their nef cones.

\begin{pro}\label{pro-HK}
Let $X$ be a compact hyper-K\"ahler manifold.  If\, $X$ satisfies the dual Kodaira condition (K), then $X$ is projective.
\end{pro}

\begin{proof}
Let $n \cnec \dim X$.  We show that if $X$ is not projective, then $\Int(\cK(X)^\vee) \cap H^{n-1,n-1}(X,\bQ) = \emptyset$. Let $V \subset H^{1,1}(X,\bR)$ be the subspace generated by $\Phi_X^{-1}(H^{n-1,n-1}(X,\bQ))$. Since $X$ is assumed to be non-projective, we have $q_X(\ga,\ga) \le 0$ for every $\ga \in H^{1,1}(X,\bQ)$; \cf~\cite[Proposition 26.13]{HuybHK}.  In particular, ${q_X}_{|V}$ is negative semi-definite.  As the signature of $q_X$ on $H^{1,1}(X,\bR)$ is $(1, h^{1,1}(X))$, there exists an $\go \in H^{1,1}(X,\bR) \bss \{0\}$ such that $q_X(\go,\go) \ge 0$ and $q_X(\go, \ga) = 0$ for every $\ga \in V$.  It follows that $\go \cdot \gb = 0$ for every $\gb \in H^{n-1,n-1}(X,\bQ)$.

\looseness=-1 Since $q_X(\go,\go) \ge 0$, up to replacing $\go$ with $-\go$, we can assume that $\go \in \ol{\cC(X)}$.  For every rational curve $C \subset X$, since $[C] \in H^{n-1,n-1}(X,\bQ)$, we have $\go \cdot C = 0$.  It follows from~\cite[Proposition 28.2]{HuybHK} that $\go \in \ol{\cK(X)}$.  Since $\go \cdot \gb = 0$ for every $\gb \in H^{n-1,n-1}(X,\bQ)$, we conclude that $\Int(\cK(X)^\vee) \cap H^{n-1,n-1}(X,\bQ) = \emptyset$.
\end{proof}

\ssec{Ricci-flat manifolds}

We finish this section with a proof of Theorem~\ref{thm-RP}.

\begin{proof}[Proof of Theorem~\ref{thm-RP}]
Let $n \cnec \dim X$.  By~\cite[Th\'eor\`eme 2]{MR730926}, there exists a finite \'etale cover $\ti{X} \to X$ such that $\ti{X} = T \times Y \times \prod_i Z_i$, where $T$ is a complex torus, $Y$ is a compact K\"ahler manifold with $H^0(Y,\gO_Y^2) = 0$, and each $Z_i$ is a hyper-K\"ahler manifold.  By Lemma~\ref{lem-finietintK}, we have $\Int(\cK(\ti{X})^\vee) \cap H^{2n-2}(\ti{X},\bQ) \ne \emptyset$.  It follows from Proposition~\ref{pro-domduK} that $\Int\(\cK(T)^\vee\) \cap H^{2\dim T-2}(T,\bQ) \ne \emptyset$ and $\Int\(\cK(Z_i)^\vee\) \cap H^{2\dim Z_i -2}(Z_i,\bQ) \ne \emptyset$.  Thus $T$ and the $Z_i$ are projective by Theorem~\ref{thm-alb} and Proposition~\ref{pro-HK}, respectively.  Since $Y$ is also projective (because $H^0(Y,\gO_Y^2) = 0$), we conclude that $X$ is projective.
\end{proof}

\section{Fibrations in abelian varieties}\label{sec-fibab}

In this section, we study the Oguiso--Peternell problem for fibrations in abelian varieties over a curve. A positive answer to the Oguiso--Peternell problem will be obtained as a consequence of the following result, which is an analog of Lemma~\ref{lem-fibTlis}.

\begin{pro}\label{pro-multsecfibab}
Let $f \colon X \to B$ be a fibration over a smooth projective curve $B$ whose general fiber is an abelian variety of dimension $g$. Assume that $X$ is a compact K\"ahler manifold and $f$\! has local sections $($for the Euclidean topology$)$ at every point of\, $B$. If there exists an $\ga \in H^{g,g}(X,\bZ)$ such that $f_*\ga \ne 0 \in H^0(B,\bZ)$, then $X$ is projective.

In particular, if $X$ satisfies the dual Kodaira condition (K), then $X$ is projective.
\end{pro}

Before we prove Proposition~\ref{pro-multsecfibab}, let us first recall some properties of the fibration $f\colon X \to B$ following~\cite[Section~4]{HYLbimkod1}; these properties partially generalize the discussion in Section~\ref{sec-fibtlis}.

Given a fibration $f\colon X \to B$ as in Proposition~\ref{pro-multsecfibab},  assume that every fiber of $f$ is a normal crossing divisor.  Since a general fiber of $f$ is Moishezon and $f$ is a proper surjective morphism between compact K\"ahler manifolds, every fiber of $f$ is Moishezon; \cf~\cite[Corollaire 2]{Campana-redalg}.  It follows from~\cite[Theorem 10.1]{CampanaPeternell2-form} that $f$ is locally projective. Let $j \colon B^\st \hto B$ be the inclusion of a non-empty Zariski open of $B$ over which $f$ is smooth, and let $\imath \colon X^\st \cnec f^{-1}(B^\st) \hto X$. Let $D \cnec f^{-1}(B \bss B^\st)$. The relative Deligne complex $\ul{D}_{X/B}(g)$ is defined to be the cone of the composition
$$R\imath_*\bZ \xto{\times (2\pi \sqrt{-1})^g}  R\imath_*\bC \simeq \gO_X^\bullet (\log{D}) \lra  \gO^{\bullet \le g-1}_{X/B}(\log{D})$$
shifted by $-1$.  Applying $Rf_*$ to the distinguished triangle
\begin{equation}\label{TD-DeligneJacsing}
  \ul{D}_{X/B}(g)  \lra R\imath_*\bZ  \lra
\gO^{\bullet \le g-1}_{X/B}(\log{D}) \lra \ul{D}_{X/B}(g)[1] 
\end{equation}
and breaking up the associated long exact sequence at $R^{2g}f_*\ul{D}_{X/B}(g)$, we obtain a short exact sequence (see~\cite[(4.6)]{HYLbimkod1})
\begin{equation}\label{SE-DeligneJacsing}
0  \lra \bar{\cJ}  \lra R^{2g}f_*\ul{D}_{X/B}(g) \lra \cH^{g,g}(X/B)  \lra 0,
\end{equation}
where $\bar{\cJ}$ is the canonical extension of the sheaf of germs of sections of the Jacobian fibration associated to $X^\st \to B^\st$ and the quotient $\cH^{g,g}(X/B)$ in~\eqref{SE-DeligneJacsing} is isomorphic to $R^{2g}(f \circ \imath)_* \bZ$; \cf~\cite[Lemma 4.6]{HYLbimkod1}.  By construction, the restriction of~\eqref{SE-DeligneJacsing} to $B^\st$ is the short exact sequence~\eqref{SE-DeligneJac} defined for the smooth torus fibration $X^\st \to B^\st$.  The natural map $\gO_X^\bullet \to \gO_{X/B}^\bullet(\log D)$ induces a map $\ul{D}_{X}(g) \to \ul{D}_{X/B}(g)$, which further induces
$$H^{2g}_D(X,\bZ(g)) = \bH^{2g}(X,\ul{D}_{X}(g)) 
\lra \bH^{2g}(X,\ul{D}_{X/B}(g)) \lra H^0(B, R^{2g}f_*\ul{D}_{X/B}(g)),$$ 
fitting into the commutative diagram 
\begin{equation}\label{SE-DeligneJaccommsing}
\begin{tikzcd}[cramped]
H^{2g}_D(X,\bZ(g)) \ar[d] \ar[r, twoheadrightarrow] & H^{g,g}(X,\bZ) \ar[d]  \\
H^0(B, R^{2g}f_*\ul{D}_{X/B}(g)) \arrow[r] & H^0(B,\cH^{g,g}(X/B))  \ar[r,"\gd"] & H^1(B, \bar{\cJ}),
\end{tikzcd}
\end{equation}
where the second row is an exact sequence induced by~\eqref{SE-DeligneJacsing} and the vertical arrow on the right is the composition
\begin{equation}\label{comp-Hggrle}
H^{g,g}(X,\bZ) \xhto{\hphantom{aaa}} H^{2g}(X,\bZ) \xto{\;f_*\;} 
H^0(B,R^{2g}f_* \bZ) \xto{\;\gamma\;}
H^0(B,R^{2g}(f \circ \imath)_* \bZ) \simeq H^0\(B,\cH^{g,g}(X/B)\).
\end{equation}
The commutative diagram~\eqref{SE-DeligneJaccommsing} is the analog of~\eqref{SE-DeligneJaccomm} with the presence of singular fibers.

Let $J \to B^\st$ be the Jacobian fibration associated to $f^\st = f_{|X^\st} \colon X^\st \to B^\st$. Let $\cE(B,J)$ be the set of bimeromorphic classes of fibrations $g \colon Y \to B$ in abelian varieties over $B$ such that
\begin{enumerate}
  	\item $g$ is locally bimeromorphically K\"ahler over $B$;
 	\item  $g$ has local sections at every point of $B$;
 	\item 	$g$ is smooth over $B^\st$, and the Jacobian fibration  
 	associated to $g^{-1}(B^\st) \xto{g} B^\st$ is $J \to B^\st$. 
 \end{enumerate}
Elements of $\cE(B,J)$ are called \textit{bimeromorphic $J$-torsors}.  There exists a map
$$\Phi \colon H^1(B,\bar{\cJ}) \lra \cE(B,J)$$
which is a generalization of the inverse of~\eqref{corr-Jtors}; \cf~the text after the proof of~\cite[Lemma~4.9]{HYLbimkod1}.
By construction, the map $\Phi$ sends $0 \in H^1(B,\bar{\cJ})$ to the bimeromorphic class of a compactification $\bar{J} \to B$ of the Jacobian fibration $J \to B^\st$ associated to $f^\st = f_{|X^\st} \colon X^\st \to B^\st$ such that the closure in $\bar{J}$ of the $0$-section of $J \to B^\st$ is a $0$-section of $\bar{J} \to B$.

Finally, for every $m \in \bZ_{>0}$ and $\eta \in H^1(B,\bar{\cJ})$, if $f \colon X \to B$ is a bimeromorphic $J$-torsor representing $\Phi (\eta)$, then there exists a bimeromorphic $J$-torsor $f_m \colon X_m \to B$ representing $\Phi (m\eta)$, together with a generically finite map $\bm \colon X \dto X_m$ over $B$; \cf~\cite[Lemma-Definition~4.12 and~p.91]{HYLbimkod1}.
The bimeromorphic $J$-torsor $f_m \colon X_m \to B$ is called the \textit{multiplication-by-$m$} of $f \colon X \to B$.

\begin{proof}[Proof of Proposition~\ref{pro-multsecfibab}]
Up to replacing $f$ with a K\"ahler log-desingularization of $(X,D)$, where $D \subset X$ is the union of singular fibers of $f$, we can assume that every fiber of $f$ is a normal crossing divisor.  Since the second row of~\eqref{SE-DeligneJaccommsing} is exact and the horizontal arrow in~\eqref{SE-DeligneJaccommsing} on the top is surjective, we have
$$\gd(\gamma(f_*\ga)) = 0 \in H^1(B,\bar{\cJ}),$$
where $f_*$ and $\gamma$ are defined in~\eqref{comp-Hggrle}.  It follows from~\cite[Lemma 4.14]{HYLbimkod1} that for some $m \in \bZ_{>0}$, the multiplication-by-$m$ of $f \colon X \to B$ is bimeromorphic to $\bar{J} \to B$.  Hence we have a generically finite map $\bm \colon X \dto \bar{J}$ over $B$, and the pre-image of the $0$-section of $\bar{J} \to B$ under $\bm$ gives a multi-section of $f\colon X \to B$.  We conclude by Corollary~\ref{cor-critcamp} that $X$ is Moishezon, hence projective (because $X$ is K\"ahler).

Finally, if $\ga \in H^{n-1,n-1}(X,\bQ) \cap \Int\(\cK(X)^\vee\)$, then $f_*\ga \ne 0 \in H^0(B,\bZ)$ by~\cite[Proposition 2.5]{OguisoPeternellconeKahdual}.  It follows from the main statement of Proposition~\ref{pro-multsecfibab} that $X$ is projective.
\end{proof}

\begin{cor}\label{cor-fibab}
Let $X$ be a compact K\"ahler manifold which satisfies the dual Kodaira condition (K).  Assume that $X$ is bimeromorphic to a compact K\"ahler manifold $X'$ which is the total space of a fibration $f \colon X' \to B$ over a smooth projective curve $B$ whose general fiber is an abelian variety. Then $X$ is projective.
\end{cor}

\begin{proof}
Let $n \cnec \dim X$.  As $X'$ is bimeromorphic to $X$ and $\Int\(\cK(X)^\vee\) \cap H^{n-1,n-1}(X,\bQ) \ne \emptyset$, there exists an $\ga \in H^{n-1,n-1}(X',\bQ) \cap \Int\(\cK(X')^\vee\)$ by Proposition~\ref{pro-domduK}.  Since the fibers of $f$ are Moishezon and $X'$ is K\"ahler, $f$ is locally projective; \cf~\cite[Theorem 10.1]{CampanaPeternell2-form}. In particular, $f$ has local multi-sections around every point of $B$, so there exists a finite morphism $r \colon \ti{B} \to B$ from a smooth curve $\ti{B}$ such that $X' \times_B \ti{B} \to \ti{B}$ has local sections around every point of $\ti{B}$. Let $\ti{X} \to X' \times_B \ti{B}$ be a K\"ahler desingularization, and let $q \colon \ti{X} \to X'$ and $\ti{f} \colon \ti{X} \to \ti{B}$ be the projections.  The situation is summarized in the commutative diagram
 $$
\begin{tikzcd}[cramped]
\ti{X}  \ar[r,"q"] \ar[d,"\ti{f}"] & X' \ar[d, "f"]  \\
\ti{B} \ar[r,"r"]  & B\rlap{.}
\end{tikzcd}
$$
Since
$r_*\ti{f}_*q^*\ga = f_*q_*q^*\ga = \deg(q) \cdot f_* \ga \ne 0$,
where the non-vanishing follows from~\cite[Proposition 2.5]{OguisoPeternellconeKahdual}, we have $\ti{f}_*q^*\ga \ne 0$. Since $\ti{f}$ has local sections around every point of $\ti{B}$ by construction, we deduce from Proposition~\ref{pro-multsecfibab} that $\ti{X}$ is projective.  Hence $X'$ and therefore $X$ are projective.
\end{proof}

\section{Elliptic fibrations}\label{sec-fibellip}

The main result of this section is the following proposition.

\begin{pro}\label{pro-ClisseOP}
Let $f \colon X \to B$ be an elliptic fibration from a smooth compact K\"ahler threefold $X$.  If\, $X$ satisfies the dual Kodaira condition (P), then $X$ is projective.
\end{pro}

Let us first prove some auxiliary results, starting with an analogous statement of Proposition~\ref{pro-ClisseOP} for surfaces.
 
\begin{lem}\label{lem-surffibP}
Let $S$ be a smooth connected compact K\"ahler surface, and let $f \colon S \to B$ be a surjective map onto a projective curve. If there exists an $\ga \in H^{1,1}(S,\bQ)$ such that $f_*\ga \ne 0 \in H^0(B,\bQ)$, then $S$ is projective.
\end{lem}

\begin{proof}
Up to replacing $\ga$ with a multiple of it, we can assume that $\ga = c_1(\cL)$ for some line bundle $\cL$ on $S$. Since $\ga \cdot [F] = f_*\ga \ne 0 \in H^0(B,\bQ) \simeq \bQ$, where $F$ is a fiber of $f$, we have $c_1(\cL \otimes \cO(mF))^2 = c_1(\cL)^2 + 2m \ga \cdot [F] >0$ for $m \gg 0$ or $m \ll 0$.  Thus $S$ is projective by~\cite[Theorem IV.6.2]{Barth}.
\end{proof}

The next lemma concerns 1-cycles vanishing away from an irreducible surface, giving a partial answer to Question~\ref{que-1cycles} in the affirmative.

\begin{lem}\label{lem-ConjNKah3}
Let $X$ be a compact K\"ahler manifold of dimension $n$ and $Y \subset X$ an irreducible surface.  Let $\ti{\imath} \colon \ti{Y} \to X$ be the composition of a desingularization of $Y$ with the inclusion $\imath \colon Y \hto X$. There exists a sub-$\bQ$-Hodge structure $L$ of $H^2(\ti{Y},\bQ)$ such that
$$H^2(\ti{Y},\bQ) = \ker\(\ti{\imath}_* \colon H^2(\ti{Y},\bQ) \lra H^{2n-2}(X,\bQ)\) \oplus L.$$
In particular, if $\ga \in H^{n-1,n-1}(X,\bQ)$ is a Hodge class which belongs to the image of\, $\ti{\imath}_*$ $($or, equivalently, vanishes in $H^{2n-2}(X \bss Y,\bQ)$ by Lemma~\ref{lem-excision}\,$)$, then $\ga = \ti{\imath}_*\gb$ for some $\gb \in H^{1,1}(\ti{Y},\bQ)$.
\end{lem}

We start with an easy lemma.

\begin{lem}\label{lem-mshsp}
Let $\phi \colon L \to M$ be a morphism of $\bQ$-Hodge structures. Assume that $L$ has a pairing $Q$ such that $(L ,Q) \otimes \bR$ is a direct sum of polarized $\bR$-Hodge structures $(L_i,Q_i)$. If\, $\ker(\phi) \subset L_i$ for some $i$, then there exists a sub-$\bQ$-Hodge structure $L' \subset L$ such that $\phi_{|L'}$ is an isomorphism onto $\Ima(\phi)$.
\end{lem}

\begin{proof}
As $Q \otimes \bR$ is a direct sum of polarizations of $\bR$-Hodge structures, the orthogonal complement $L' \cnec \ker(\phi)^\perp$ of $\ker(\phi)$ with respect to $Q$ is a sub-$\bQ$-Hodge structure of $L$, and we have $\dim \ker(\phi)^\perp + \dim \ker(\phi) = \dim L$ (because $Q \otimes \bR$ is non-degenerate).  Thus to prove Lemma~\ref{lem-mshsp}, it suffices to show that $\ker(\phi) \cap \ker(\phi)^\perp = 0$, or equivalently $\ker(\phi_\bR) \cap \ker(\phi_\bR)^\perp = 0$.

Since $(L_i,Q_i)$ is a polarized $\bR$-Hodge structure, the orthogonal complement $L'_i$ of $\ker(\phi_\bR)$ with respect to $Q_i$ satisfies $\ker(\phi_\bR) \cap L'_i = 0$. Hence
\begin{equation*}\pushQED{\qed}
\ker(\phi_\bR) \cap \ker(\phi_\bR)^\perp = \ker(\phi_\bR) \cap \(L_i' \oplus \bigoplus_{j \ne i} L_j\) = 0.\qedhere \popQED
	\end{equation*}
\renewcommand{\qed}{}     
\end{proof}

\begin{proof}[Proof of Lemma~\ref{lem-ConjNKah3}]
Fix a K\"ahler class $\go \in H^2(X,\bR)$, and let $H =(\ti{\imath}^*\go)^\perp \subset H^2(\ti{Y},\bR)$, where the orthogonal is defined with respect to the intersection product on $H^2(\ti{Y},\bR)$. Since $\ti{\imath}^*\go^2 \ne 0$, we have
 $$H^2(\ti{Y},\bR) = \bR \ti{\imath}^*\go \oplus H.$$ 
We verify that the restriction of the intersection product to $H$ is a polarization of the $\bR$-Hodge structure $H$: the induced pairing on $H$ is non-degenerate, and the Hodge decomposition $H \otimes \bC = H^{2,0} \oplus H^{1,1} \oplus H^{0,2}$ is orthogonal with respect to the Hermitian form $h (\gb,\gamma) \cnec \gb \cdot \bar{\gamma}$.  The restriction of $h$ to $H^{2,0} \oplus H^{0,2}$ is positive definite.  Since $Y$ is irreducible, $H^{1,1}(\ti{Y})$ is of signature $(1, \dim H^{1,1}(\ti{Y}) - 1)$.  As $\ti{\imath}^*\go^2 > 0$, the restriction of $h$ to $H^{1,1}$ is negative definite.

We have $\ker({\ti{\imath}}_*) \subset H$.  Indeed, let $\xi \in \ker({\ti{\imath}}_*)$, and write $\xi = a \cdot {\ti{\imath}}^*\go + \gb$ with $a \in \bR$ and $\gb \in H$. Then
$$0 =   {\ti{\imath}}_*(a \cdot {\ti{\imath}}^*\go + \gb) \cdot \go = a[Y]\cdot \go^2 + {\ti{\imath}}_*(\gb \cdot  {\ti{\imath}}^*\go) = a[Y]\cdot \go^2.$$
As $[Y]\cdot \go^2 \ne 0$, we have $a = 0$, so $\xi \in H$.  It follows from Lemma~\ref{lem-mshsp} that there exists a sub-$\bQ$-Hodge structure $L$ of $H^2(\ti{Y},\bQ)$ such that $H^2(\ti{Y},\bQ) = \ker({\ti{\imath}}_*) \oplus L$, which proves the main statement of Lemma~\ref{lem-ConjNKah3}.  The last statement of Lemma~\ref{lem-ConjNKah3} follows from the observation that $\ti{\imath}_{*|L} \colon L \to H^{2n-2}(X,\bQ)$ is an isomorphism of $\bQ$-Hodge structures onto $\Ima \(\ti{\imath}_*\)$.
\end{proof}

The bimeromorphic models of elliptic threefolds in the following statement will be useful in our proof of Proposition~\ref{pro-ClisseOP}.

\begin{thm}[\cf~{\cite[Theorem A.1]{NakayamaLoc}}]\label{thm-Nakmodbim}
Let $f\colon X \to B$ be an elliptic fibration with $X$ being a compact K\"ahler threefold.  Then $f$\! is bimeromorphic to an elliptic fibration $f' \colon X' \to B'$ over a compact normal surface $B'$ satisfying the following properties:
	\begin{itemize}
		\item $X'$ has at worst terminal singularities.
		\item $f'$ has equidimensional fibers.
	\end{itemize}
\end{thm}

\begin{proof}
Up to replacing $X$ and $B$ with some K\"ahler desingularizations of them, we can assume that both $X$ and $B$ are smooth, $X$ is K\"ahler, and the discriminant locus of $f$ is a normal crossing divisor of $B$.  Under these assumption, $f$ is locally projective by~\cite[Theorem 3.3.3]{NakayamaLoc}.  Theorem~\ref{thm-Nakmodbim} is then a consequence of~\cite[Theorem A.1]{NakayamaLoc}.
\end{proof}

Now we prove Proposition~\ref{pro-ClisseOP}.

\begin{proof}[Proof of Proposition~\ref{pro-ClisseOP}]
Let $f' \colon X' \to B'$ be an elliptic fibration bimeromorphic to $f$ as in Theorem~\ref{thm-Nakmodbim}.  Note that since $\dim X' = 3$ and $\dim B' = 2$, both $X'$ and $B'$ have at worst isolated singularities.  Let
	$$Z \cnec \Sing(B') \cup f'(\Sing(X')),$$ 
which is a finite subset of $B'$.
	
By Corollary~\ref{cor-dompseudu}, we can freely replace $X$ with any smooth bimeromorphic model of it.  Up to replacing $f \colon X \to B$ with a bimeromorphic model of it by resolving the bimeromorphic map $X \dto X'$ and taking K\"ahler desingularizations, we can assume that both ${X}$ and $B$ are compact K\"ahler manifolds, and $f$ fits into the commutative diagram
	\begin{equation}\label{cd-modtermres}
	\begin{tikzcd}[cramped]
	{X} \ar[r,"\nu"]\ar[d," {f}"]  & X'  \ar[d,"f'"]  \\
	{B} \ar[r,"\mu"] & B'\rlap{,} 
	\end{tikzcd}
	\end{equation}
where the horizontal arrows are bimeromorphic morphisms.
	
Since $\Int(\Psef( {X})^\vee)$ contains a rational class $\ga$, we have $ {f}_*\ga \in \Int(\Psef( {B})^\vee)$ by Lemma~\ref{lem-push}.  In particular, the surface $B$ satisfies the dual Kodaira condition, so $B$ is projective.  Thus $ {f}_*\ga \in H^{1,1}( {B},\bQ)$ is ample by Kleiman's criterion.  Up to replacing $\ga$ with a positive multiple of it, we can assume that $ {f}_*\ga = c_1(\cL) \in H^{1,1}( {B},\bQ)$ for some very ample line bundle $\cL$ on $ {B}$.  Up to further replacing $\ga$ with a multiple of it, we can find a linear system $T \subset |\cL|$ such that a general member $C$ of $T$ satisfies the following properties:
	\begin{itemize}
		\item $C$ is smooth and irreducible.
		\item $C' \cnec \mu(C)$ is a curve containing $Z \subset B'$ (as $Z$ is finite).
		\item $D' \cnec f'^{-1}(C')$ is irreducible (as $f'$ is equidimensional).
		\item A general pair of points of $ {B}$ is connected by a chain of general members of $T$.
	\end{itemize}

Let $C$ be a general member of $T$ (satisfying the above properties).  We have the commutative diagram
\begin{equation}\label{cd-fibell1}
\begin{tikzcd}[cramped]
H_2(D',\bQ)  \ar[r , "\imath'_*"] \ar[d, "(f'_{|D'})_*"] &  H_2(X',\bQ)  \ar[r] \ar[d, "f'_*"] & H^{\BM}_2(X' \bss D',\bQ)  \ar[d, "\wr"', "(f'_{|X'\bss D'})_*"]  \\
H_2(C',\bQ)  \ar[r] & H_2(B',\bQ)  \ar[r] &  H^{\BM}_2(B' \bss C',\bQ)\rlap{,}  
\end{tikzcd}
\end{equation}
where the rows of~\eqref{cd-fibell1} are part of the long exact sequences of Borel--Moore homology groups induced by the open embeddings, \cf~\cite[Section~IX.2, (2.1)]{IversenBook},
and the vertical arrows between them are induced by $f'\colon X' \to B'$.

\begin{lem}\label{lem-isomf'x'd'}
	The map $(f'_{|X'\bss D'})_*$ in~\eqref{cd-fibell1} is an isomorphism. 
\end{lem}

\begin{proof}
As $C' = \mu(C)$ contains $Z \cnec \Sing(B') \cup f'(\Sing(X'))$, both $X'\bss D'$ and $B'\bss C'$ are smooth.  So $(f'_{|X'\bss D'})_*$ is identified with
$$\left(f'_{|X'\bss D'}\right)_* \colon H^4(X' \bss D',\bQ) \lra H^2(B' \bss C',\bQ)$$
through Poincar\'e duality.  Let $\ti{C} \cnec \mu^{-1}(C') \subset {B}$.  Then $\ti{C} = C \cup C_1$ for some curve $C_1 \subset {B}$.  As $m(m'C + C_1)$ is very ample for $m, m' \gg 0$, the complement $B' \bss C' \simeq {B} \bss \ti{C} = {B} \bss (C \cup C_1)$ is affine.  As $f'$ has equidimensional connected fibers, it follows from Lemma~\ref{lem-isomGysin} that $(f'_{|X'\bss D'})_*$ is an isomorphism.
\end{proof}

Let $\ti{D}$ be a desingularization of the proper transform of the divisor $D'$ under $\nu$, and let
$$
\begin{tikzcd}[cramped]
\ti{D} \ar[r," {\imath}"] \ar[d] &  {X}  \ar[d,"\nu"]  \\
D'  \ar[r ,hook, "\imath'"] & X'
\end{tikzcd}
$$
be the induced commutative diagram. Since
$$f'_*\nu_*\ga = \mu_* {f}_*\ga = [C'] \in H_2(B',\bQ),$$
by~\eqref{cd-fibell1} and Lemma~\ref{lem-isomf'x'd'} we have $\nu_*\ga \in \Ima\(\imath'_* \colon H_2(D',\bQ) \to H_2(X',\bQ)\)$.  As $\ga \in H_2( {X},\bQ)$, which is a pure Hodge structure of weight $-2$, we have $\nu_*\ga \in \bW_{-2}H_2(X',\bQ)$, so by the strictness of $\nu_*$, we have $\nu_*\ga = \imath'_* \ga_0$ for some $\ga_0 \in \bW_{-2}H_2(D',\bQ)$.  Since $\ti{D} \to D'$ is a desingularization of $D'$, $\ga_0 \in \bW_{-2}H_2(D',\bQ)$ can be lifted to $\ga_1 \in H_2(\ti{D},\bQ)$ by Lemma~\ref{lem-fpods}, so
$$\nu_* {\imath}_*\ga_1 = \nu_*\ga \in H_2(X',\bQ).$$
Since $X'$ has at worst rational singularities, Lemma~\ref{lem-clHdgctr} implies that $ {\imath}_*\ga_1 - \ga \in H_2( {X},\bQ)$ is a Hodge class.  As $\ga$ is a Hodge class, so is $ {\imath}_*\ga_1$.  Since $\ti{D}$ is irreducible, by Lemma~\ref{lem-ConjNKah3} there exists an $\ga_2 \in H^{1,1}(\ti{D},\bQ)$ such that $ {\imath}_*\ga_2 = {\imath}_*\ga_1$.  By construction, $ {f}( {\imath}(\ti{D})) = C$, so we have the factorization
$$ {f} \circ  {\imath} \colon  \ti{D} \xto{\;\;p\;} C \xhto{\;\;\jmath\;\;}  {B}.$$
We have 
$$\mu_* {f}_* {\imath}_*\ga_2 = 
f'_*\nu_* {\imath}_*\ga_2 = f'_*\nu_* {\imath}_*\ga_1 = f'_*\nu_*\ga = \mu_* {f}_*\ga  = [C'] \ne 0,$$
so
$$\jmath_*p_*\ga_2 =  {f}_* {\imath}_*\ga_2 \ne 0,$$
and thus $p_*\ga_2 \ne 0 \in H^0(C,\bQ)$.  It follows from Lemma~\ref{lem-surffibP} that $\ti{D}$ is a projective surface, so $D \cnec \imath(\ti{D}) \subset X$ is algebraically connected.

As a general pair of points of $ {B}$ is connected by a chain of general members $C$ of $T$, when $C$ varies in $T$, the divisors $D$ thus connect $ {X}$ by construction.  Hence ${X}$ is projective by Theorem~\ref{thm-critcamp}.
\end{proof}

Thanks to Proposition~\ref{pro-ClisseOP}, we can exclude threefolds of algebraic dimension 2 in Problem~\ref{prob-psefdual}.

\begin{cor}\label{cor-a=2}
Let $X$ be a smooth compact K\"ahler threefold.  If\, $X$ satisfies the dual Kodaira condition (P), then $a(X) \ne 2$.
\end{cor}

\begin{proof} 
Assume that $a(X) = 2$. Then the algebraic reduction of $X$ is bimeromorphic to an elliptic fibration $f \colon X' \to B$; \cf~\cite[Theorem 12.4]{UenoClassAlgVar}.  By desingularization, we can assume that $X'$ is smooth and K\"ahler.  As $\Int\(\Psef(X)^\vee\)$ contains a rational class, $\Int\(\Psef(X')^\vee\)$ contains a rational class $\ga$ as well by Corollary~\ref{cor-dompseudu}.  Thus $X'$ is projective by Proposition~\ref{pro-ClisseOP}, which contradicts $a(X) = 2$. Hence $a(X) \ne 2$.
\end{proof}

\section{Proofs of Theorems~\ref{thm-mainOP} and~\ref{thm-mainOPK}}\label{sec-preuve}

\begin{proof}[Proof of Theorem~\ref{thm-mainOPK}]
Let $X$ be a compact K\"ahler threefold as in Theorem~\ref{thm-mainOPK}.  Assume to the contrary that $a(X) \le 1$. Then $X$ is bimeromorphic to a variety $X'$ satisfying one of the descriptions listed in Proposition~\ref{pro-class}.  We will rule out these descriptions case by case and therefore obtain a contradiction.  Suppose that $X'$ is in case~\eqref{pc-1}, namely $X'$ is a $\bP^1$-fibration $X' \to S$ over a smooth compact K\"ahler surface.  Then $S$ is projective by~\cite[Proposition 2.6]{OguisoPeternellconeKahdual}, so $X'$ is Moishezon by Corollary~\ref{cor-critcampP1}, which is impossible.  If $X'$ is in case~\eqref{pc-2}, then the projection $(S \times B)/G \to S/G$ induces a dominant meromorphic map $X \dto S/G$. Once again,~\cite[Proposition 2.6]{OguisoPeternellconeKahdual} implies that $S/G$ is projective, contradicting the fact that $S$ is non-algebraic.  Cases~\eqref{pc-3} and~\eqref{pc-4} are ruled out by Corollary~\ref{cor-fibab} and Proposition~\ref{pro-fibTlis}, respectively.  Finally, we rule out case~\eqref{pc-5} by Corollary~\ref{cor-T/G}.
\end{proof}

\begin{proof}[Proof of Theorem~\ref{thm-mainOP}]
Let $X$ be a compact K\"ahler threefold as in Theorem~\ref{thm-mainOP}.  Since $\cK(X) \subset \Psef(X)$, Theorem~\ref{thm-mainOPK} already implies that $a(X) \ge 2$. The case $a(X) = 2$ is excluded by Corollary~\ref{cor-a=2}; hence $X$ is projective.
\end{proof}

\section{One-cycles in compact K\"ahler threefolds and the Oguiso--Peternell problem}\label{sec-1cycles}

In this final section, we work under the assumption that Question~\ref{que-1cycles} has a positive answer and prove that every compact K\"ahler threefold $X$ as in Problem~\ref{prob-OP} or~\ref{prob-OPCampl} is projective (\cf~Corollary~\ref{cor-OPpar}), except for simple non-Kummer threefolds which presumably do not exist (see Remark~\ref{rem-simpNK}).  Already by Theorem~\ref{thm-mainOPK}, we know that such a threefold $X$ has algebraic dimension $a(X) \ge 2$, so the proof consists in excluding the case $a(X) = 2$.  We will deduce the latter as a consequence of the following result.

\begin{pro}\label{pro-big}
Let $f \colon X \to B$ be an elliptic fibration over a smooth projective surface $B$. Suppose that $X$ is a smooth compact K\"ahler threefold and that Question~\ref{que-1cycles} has a positive answer for $X$. If there exists an $\ga \in H^{2,2}(X,\bQ)$ such that $f_*\ga \in H^{1,1}(B,\bQ)$ is big, then $X$ is projective.
\end{pro}

\begin{proof}
As $f_*\ga \in H^{1,1}(B,\bQ)$ is big, $f_*\ga$ is the sum of an ample curve class and an effective curve class; \cf~\cite[Corollary 2.2.7]{LazarsfeldPosI}. We have the following more precise statement.
	
\begin{lem}\label{lem-bigdecomp}
Up to replacing $\ga$ by a positive multiple of it, there exist a very ample line bundle $\cL$, integral curves $E_1,\ldots,E_l \subset B$, and positive integers $n_1,\ldots, n_l$ such that
$$f_*\ga = c_1(\cL) + \sum_{i=1}^l n_i[E_i] \in H^{1,1}(B,\bQ)$$ and that $c_1(\cL)$ does not lie in the subspace of $H^2(B,\bQ)$ spanned by $[E_1],\ldots, [E_l]$.
\end{lem}
	
\begin{proof}
Since $f_*\ga \in H^{1,1}(B,\bQ)$ is big, up to replacing $\ga$ by a positive multiple of it, there exist a very ample line bundle $\cL'$, integral curves $E_1,\ldots,E_l \subset B$, and $n'_1,\ldots, n'_l \in \bZ_{>0}$ such that
$$f_*\ga = c_1(\cL') + \sum_{i=1}^l n'_i[E_i] \in H^{1,1}(B,\bQ).$$
Suppose that $ c_1(\cL') = \sum_{i=1}^l m_i[E_i]$ for some $m_i \in  \bQ $. Then
		\begin{equation}\label{eqn-comblinpre}
		m_l \cdot f_*\ga = (n'_l + m_l)c_1(\cL') +  \sum_{i=1}^{l-1} \(m_ln'_i  - n'_l m_i\) [E_i].
		\end{equation}
Up to reordering the indices, we can assume that $\frac{m_l}{n'_l} \ge \frac{m_i}{n'_i}$ (so $m_ln'_i - n'_l m_i \ge 0$) for every $i = 1,\ldots, l$.  As $\cL'$ is ample, we have $m_j > 0$ for at least one $j$, so $m_l > 0$. Therefore, up to replacing $\ga$ by a positive multiple of it,~\eqref{eqn-comblinpre} gives a new expression
		\begin{equation}\label{eqn-comblin'}
		f_*\ga = c_1(\cL'') + \sum_{i=1}^{l'} n''_i[E_i] \in H^{1,1}(B,\bQ)
		\end{equation}
for some $l' < l$ and $n_1,\ldots, n_{l'} \in \bZ_{>0}$ together with another very ample line bundle $\cL''$.
		
We can repeat the same procedure as long as $c_1(\cL'') \in \sum_{i=1}^{l'} \bQ [E_i]$. Since the integer $l'$ in~\eqref{eqn-comblin'} decreases strictly, this procedure eventually stops, which gives an expression $f_*\ga = c_1(\cL) + \sum_{i=1}^l n_i[E_i]$ satisfying the properties in Lemma~\ref{lem-bigdecomp}.
\end{proof}
	
Let $Z \subset B$ be the subset such that $\dim f^{-1}(z) > 1$ for every $z \in Z$.  As $f$ is a surjective morphism from an irreducible threefold to a surface, $Z$ is finite.
	
We write $f_*\ga = c_1(\cL) + \sum_{i=1}^l n_i[E_i]$ as in Lemma~\ref{lem-bigdecomp}.  Up to further replacing $\ga$ with a multiple of it, we can find a linear system $T \subset |\cL|$ such that a general member $C$ of $T$ satisfies the following properties: 
	\begin{itemize}
		\item $C$ is smooth and irreducible.
		\item $C \ne E_i$ for every $i$, and 
		$C$ contains $Z \subset B$ (as $Z$ is finite). 
		\item A general pair of points of ${B}$ is connected by a chain of general members of $T$.
	\end{itemize}
Let $C$ be a general member of $T$, and let $E \cnec \cup_i E_i$.  Let $D \cnec f^{-1}(C)$ and $D' \cnec f^{-1}(E)$.  We have the commutative diagram
	\begin{equation}\label{cd-fibell2}
	\begin{tikzcd}[cramped]
	H_2(D \cup D',\bQ)  \ar[r , "\imath_*"] \ar[d] &  H_2(X,\bQ)  \ar[r] \ar[d, "f_*"] & H^{\BM}_2(X \bss (D \cup D'),\bQ)  \ar[d, "\wr"', "(f_{|X\bss (D \cup D')})_*"]  \\
	H_2(C \cup E,\bQ)  \ar[r] & H_2(B,\bQ)  \ar[r] &  H^{\BM}_2(B \bss (C\cup E),\bQ)\rlap{,}
	\end{tikzcd}
	\end{equation}
        where the rows of~\eqref{cd-fibell2} are part of the long exact sequences of Borel--Moore homology groups induced by the open embeddings, \cf~\cite[Section~IX.2, (2.1)]{IversenBook},
        and the vertical arrows between them are induced by $f \colon X \to B$.  The same argument proving Lemma~\ref{lem-isomf'x'd'} shows that the map $(f_{|X\bss (D \cup D')})_*$ in~\eqref{cd-fibell2} is an isomorphism.
		
Let $\ti{D}$ (resp.\ $\ti{D}'$) be a desingularization of $D$ (resp.\ $D'$), and let
	$$
	\ti{\imath} \colon \ti{D} \sqcup \ti{D}' \lra D \cup D' \xhto{\hphantom{aaa}} X
	$$
be the composition of the desingularizations and inclusion.  Since $f_*\ga \in H_2(B,\bQ)$ is supported on $C \cup E$, by~\eqref{cd-fibell2} we have
	$$\ga \in \Ima\(\imath_*  \colon H_2(D \cup D',\bQ) \lra H_2(X,\bQ)\).$$
As we assume that Question~\ref{que-1cycles} has a positive answer for $X$, there exist $\gb \in H^{1,1}(\ti{D},\bQ)$ and $\gb' \in H^{1,1}(\ti{D}',\bQ)$ such that $ \ti{\imath}_*(\gb + \gb') = \ga$.  By construction, we have the factorization
	$$ {f} \circ  \ti{\imath} \colon  \ti{D} \sqcup \ti{D}' 
	\!\xto{\;p \sqcup p'} C \cup E \xhto{\;\;\jmath\;\;}  {B},$$
so 
	\begin{equation}\label{eqn-jpgb}
	\jmath_*p_*\gb + \jmath_*p'_*\gb' =  {f}_* \ti{\imath}_*(\gb + \gb')  
	=  {f}_*\ga = c_1(\cL) + \sum_{i=1}^l n_i[E_i]
	\end{equation}
in $H^2(B,\bQ)$. As $\jmath_*p_*\gb$ is supported on $C$ and $\jmath_*p'_* \gb'$ supported on $\cup_i E_i$, we have
	$$\jmath_*p_*\gb \in \bQ \cdot [C] = \bQ \cdot c_1(\cL) \subset H^2(B,\bQ) \quad
	\text{and} \quad \jmath_*p'_*\gb' \in \sum_i \bQ \cdot [E_i] \subset H^2(B,\bQ).$$
Since $c_1(\cL) \notin \sum_i \bQ \cdot [E_i]$ by Lemma~\ref{lem-bigdecomp}, it follows from~\eqref{eqn-jpgb} that $\jmath_*p_*\gb \ne 0$ and thus $p_*\gb \ne 0 \in H^0(C,\bQ)$.  Therefore, $\ti{D}$ is projective by Lemma~\ref{lem-surffibP}, so $D$ is algebraically connected.
	
As a general pair of points of $ {B}$ is connected by a chain of general members $C$ of $T$, when $C$ varies in $T$, the divisors $D$ connects $ {X}$ by construction. It follows from Theorem~\ref{thm-critcamp} that ${X}$ is projective.
\end{proof}

\begin{proof}[Proof of Corollary~\ref{cor-OPpar} for Problem~\ref{prob-OP}]
As we mentioned before, Theorem~\ref{thm-mainOPK} implies that $a(X) = 2$, so $X$ is bimeromorphic to an elliptic fibration $f \colon X' \to B$ over a projective surface.  By desingularization, we can assume that both $X'$ and $B$ are smooth and $X'$ is K\"ahler.  As $\Int\(\cK(X)^\vee\)$ contains a rational class, $\Int\(\cK(X')^\vee\)$ contains a rational class $\ga$ as well, \cf~\cite[Proposition 2.1]{OguisoPeternellconeKahdual}, and we have $f_*\ga \in \Int\(\cK(B)^\vee\)$ by~\cite[Proposition 2.5]{OguisoPeternellconeKahdual}.  Since $B$ is a smooth projective surface, $f_*\ga \in H^{1,1}(B,\bQ)$ is big by Kleiman's criterion.  Applying Proposition~\ref{pro-big} to the elliptic fibration $f \colon X' \to B$ shows that $X'$ is projective. Hence $X$ is projective.
\end{proof}

Finally, we prove Corollary~\ref{cor-OPpar} for Problem~\ref{prob-OPCampl}. Before we start the proof, let us first recall and prove some statements about subvarieties with ample normal bundles. The first one is a theorem due to Fulton and Lazarsfeld, asserting that a subvariety with ample normal bundle intersects non-negatively with other subvarieties.

\begin{thm}[Fulton--Lazarsfeld {\cite[Corollary 8.4.3]{Lazarsfeld2004a}}]\label{thm-FLpos}
Let $X$ be a compact complex manifold, and let $Y \subset X$ be a local complete intersection subvariety of dimension $k$.  Assume that $N_{Y/X}$ is ample. Then for every subvariety $Z \subset X$ of codimension $k$, we have $Y \cdot Z \ge 0$. Moreover, if $Y \cap Z \ne \emptyset$, then $Y \cdot Z > 0$.
\end{thm}

While the pullback of an effective cycle does not necessarily remain effective, in some situations the ampleness of the normal bundle of a subvariety $Y \subset X$ ensures that the pullback of the cycle class of $Y$ is still effective.

\begin{lem}\label{lem-bimcampeff}
Let $\mu \colon Y \to X$ be a bimeromorphic morphism between smooth compact K\"ahler threefolds, and let $C \subset X$ be a smooth irreducible curve. If\, $N_{C/X}$ is ample, then $\mu^*[C] \in H^4(Y,\bQ)$ is an effective curve class. More precisely, there exists an irreducible curve $\ti{C}$ on $Y$ such that $\mu(\ti{C}) = C$ and $m\mu^*[C] = [\ti{C}]+[C']$ for some $m \in \bZ_{>0}$ and some effective curve class $[C']\in H^4(Y,\bQ)$.
\end{lem}

We need to first prove a technical lemma before proving Lemma~\ref{lem-bimcampeff}.  Let $\sX$ be a compact complex variety.  A curve $\sC \subset \sX$ is called \textit{displaceable} if for every point $x \in \sX$ and every irreducible component $\sC'$ of $\sC$, we can find a curve $\sC'' \subset \sX$ such that $\sC''$ is deformation equivalent to $\sC'$ and $x \notin \sC''$. For any subvariety $\sY \subset \sX$, we say that $\sC$ is \textit{displaceable in $\sY$} if $\sC\subset \sY$ and in the previous definition, $\sC''$ is deformation equivalent to $\sC'$ in $\sY$.

\begin{lem}\label{lem-tech}
Let $\nu \colon \ti{\sX} \to \sX$ be the blow-up of a compact K\"ahler threefold $\sX$ along an irreducible smooth center $Z \subset \sX$.  Let $\sC \subset \sX$ be a displaceable curve.  Then $\nu^*[\sC] \in H^4(\ti{\sX},\bQ)$ can be represented by a displaceable curve $\ti{\sC} \subset \ti{\sX}$.
	
If moreover $\sC$ has an irreducible component $\sC_0$ which is displaceable in some surface $S \subset \sX$, then $\nu^*[\sC] \in H^4(\ti{\sX},\bQ)$ can be represented by a displaceable curve $\ti{\sC} \subset \ti{\sX}$ which contains an irreducible component $\ti{\sC}_0$ such that $\nu(\ti{\sC}_0)$ is deformation equivalent to $\sC_0$ in $S$ and $\ti{\sC}_0$ is displaceable in the strict transform $\ti{S} \subset \ti{\sX}$ of $S$.
\end{lem}

\begin{proof}
Let $\sC_0,\ldots,\sC_k$ be the irreducible components of $\sC$.  The blow-up center $Z$ is either a point or a curve.  If $Z$ is a point (resp.\ a curve), then since $\sC$ is displaceable, we can choose a general deformation $\sC_i' \subset \sX$ of $\sC_i$ such that $\sC_i' \cap Z = \emptyset$ (resp.\ $\sC_i' \cap Z = \emptyset$ or $\dim \sC_i' \cap Z = 0$). For the irreducible component $\sC_0$, we choose a general deformation $\sC'_0$ of $\sC_0$ which moreover remains in $S$.  By the blow-up formula~\cite[Theorem~6.7]{Fulton}, we have
$$\nu^*[\sC] = \sum_i \nu^*[\sC_i'] = \sum_i [\ti{\sC}_i'] + m[F] \in H^4(\ti{\sX},\bQ)$$
for some $m \in \bZ_{\ge 0}$, where $\ti{\sC}_i' \subset \ti{\sX}$ is the strict transform of $\sC_i'$ and $F$ is a fiber of $\nu^{-1}(Z) \to Z$.

Clearly, $F$ is displaceable.  Since each $\sC_i'$ (resp.\ $\sC_0'$) is a general deformation of the displaceable curve $\sC_i$ (resp.\ $\sC_0$), the strict transform $\ti{\sC}_i'$ is displaceable (resp.\ displaceable in $\ti{S}$).  Finally, since $\ti{\sC}_0'$ is the strict transform of ${\sC}'_0$ under $\nu$, the image $\nu(\ti{\sC}_0') = {\sC}'_0$ is deformation equivalent to $\sC_0$ in $S$ by assumption.
\end{proof}

\begin{proof}[Proof of Lemma~\ref{lem-bimcampeff}]
Let $q \colon X' \to X$ be the blow-up of $X$ along $C$. We resolve the bimeromorphic map $\mu^{-1} \circ q \colon X' \dto Y$ by a sequence of blow-ups along smooth centers $\nu \colon \ti{X} \to X'$ and let $p \colon \ti{X} \to Y$ be the induced bimeromorphic morphism. The following commutative diagram summarizes the situation: 
$$
\begin{tikzcd}[cramped]
\ti{X} \ar[r , "\nu"] \ar[d,"p"] & X' \ar[d,"q"] \\
Y \ar[r , "\mu"] & X\rlap{.}
\end{tikzcd}
$$
Since $\mu^*[C] = p_* \nu^*q^* [C] $, it suffices to prove Lemma~\ref{lem-bimcampeff} for the bimeromorphic morphism $q \circ \nu$.
 
Let $E = q^{-1}(C)$ and $g = {q}_{|E} \colon E \to C$.  By the blow-up formula~\cite[Proposition 6.7(a)]{Fulton}, we have
 \begin{equation}\label{eqn-blqC}
 q^*[C] = j_*c_1\(g^*N_{C/X}\) + j_*c_1\(\cO_{E/C}(1)\),
 \end{equation} 
where $j \colon E \hto X'$ is the inclusion. As $N_{C/X}$ is ample, $c_1(g^*N_{C/X}) + c_1(\cO_{E/C}(1))$ is an ample class in $E$.  Let $m \in \bZ_{>0}$ be such that $\cL \cnec \(\det(g^*N_{C/X}) \otimes \cO_{E/C}(1) \)^{\otimes m}$ is very ample.  We can therefore find an irreducible curve $C_1 \in |\cL|$ which is displaceable in $E$.  As $\nu$ is a sequence of blow-ups of threefolds along smooth centers, applying Lemma~\ref{lem-tech} to these blow-ups yields an irreducible curve $\ti{C} \subset \ti{X}$ such that $\nu(\ti{C})$ is deformation equivalent to $C_1$ in $E$ and $\nu^*[C_1] = [\ti{C}] + [C']$ for some effective curve class $[C']$ in $\ti{X}$.  Since $m\cdot q^*[C] = [C_1]$ by~\eqref{eqn-blqC} and the construction of $C_1$, we have
$m\cdot \nu^*q^*[C] = [\ti{C}] + [C']$.
Finally, since $\nu(\ti{C})$ is deformation equivalent to $C_1$ in $E$ and $g(C_1) = C$, necessarily $q(\nu(\ti{C})) = g(\nu(\ti{C})) = C$.  This proves Lemma~\ref{lem-bimcampeff} for the bimeromorphic morphism $q \circ \nu$.
\end{proof}

\begin{proof}[Proof of Corollary~\ref{cor-OPpar} for Problem~\ref{prob-OPCampl}]

By~\cite[Corollary 4.8]{OguisoPeternellconeKahdual}, we have $a(X) \ge 2$.

Assume that $a(X) = 2$. By Lemma~\ref{lem-redalgph}, the algebraic reduction $f \colon X \dto B$ of $X$ is almost holomorphic, and its general fiber $F$ is an elliptic curve. Let $\gS$ be the irreducible component of the Douady space of $X$ containing $F$, and let
$$
\begin{tikzcd}[cramped]
\cC \ar[r , "q"] \ar[d,"p"] & X \\
\gS
\end{tikzcd}
$$
denote the universal family. Since $X$ is a compact K\"ahler manifold, $\gS$ is compact (\cf~\cite{FujikiClosednessDouady}), and so is~$\cC$. As
$$2 = \dim B \le \dim \gS \le H^0(F,N_{F/X}) = H^0(F,\cO^2_F) =  2,$$ 
the generically injective meromorphic map $\tau \colon B \dto \gS$ induced by the almost holomorphic fibration $X \dto B$ is bimeromorphic. Consequently, $q$ is bimeromorphic.

Let $\nu \colon \gS' \to \gS$ be the normalization of $\gS$, and let $\cC' \cnec \cC \times_{\gS} \gS'$.  We have the commutative diagram
$$
\begin{tikzcd}[cramped]
\ti{\cC} \ar[r, swap, "\ti{\nu}"] \ar[rrr , bend left=40, "\ti{q}"]  \ar[d,"\ti{p}"] & 
\cC' \ar[r ] \ar[d,"p'"]  \ar[rr , bend left=20, "q'"]  \arrow[dr, phantom, "\square"] & 
\cC \ar[r , swap, "q"] \ar[d,"p"] & X \\
\ti{\gS} \ar[r , "\nu'"] & {\gS}' \ar[r, "\nu"] & \gS\rlap{,}
\end{tikzcd}
$$
where $\nu'$ (resp.\ $\ti{\nu}$) is a desingularization of $\gS'$ (resp.\ $\cC'$).  Since $C$ is a smooth curve and $N_{C/X}$ is ample, the irreducible components of $C$ have ample normal bundles as well, so we can assume that $C$ is irreducible.  By Lemma~\ref{lem-bimcampeff}, there exist an $m \in \bZ_{>0}$ and an irreducible curve $\ti{C} \subset \ti{\cC}$ such that $\ti{q}(\ti{C}) = C$ and $m\ti{q}^*[C] = [\ti{C}] + [C'] \in H^4(\ti{\cC},\bQ)$ for some effective curve class $[C']$.
 
Now we show that $\ti{p}_*([\ti{C}] + [C'])$ is big.  Since $\ti{p}_*([\ti{C}] + [C'])$ is effective, it suffices to show that $\(\ti{p}_*([\ti{C}] + [C'])\)^2 > 0$.  First we note that if $F'$ is a fiber of $p'$, then $C \nsubset q'(F')$.  Indeed, if $C \subset q'(F')$, then since $p'$ is flat and a general fiber of $p$ is an elliptic curve, we would have $g(C) \le 1$. So $X$ would be projective by~\cite[Corollary 4.6]{OguisoPeternellconeKahdual}, contradicting $a(X) = 2$.  It follows from $C \nsubset q'(F')$ that $\nu'(\ti{p}(\ti{C})) \subset \gS'$ is a curve.

\begin{claim} 
  The analytic subset
  $D \cnec \ti{q} \( \ti{p}^{-1}( \ti{p}(\ti{C}))\)$ has a divisorial component $D'$ containing $C$.
\end{claim}

\begin{proof}
Let $E \subset \ti{\gS}$ be the exceptional divisor of $\nu'$, and let $U \cnec \ti{\gS} \bss E$.  Since $\nu'$ is a bimeromorphic morphism between normal surfaces and since $\nu'(\ti{p}(\ti{C})) \subset \gS'$ is a curve, we have $C^\circ \cnec \ti{p}(\ti{C}) \cap U \ne \emptyset$, which is a curve, and $\nu(\nu'(C^\circ))$ is a curve as well.  As $p \colon \cC \to \gS$ is the universal family of the Douady space $\gS$, the complex subspace
 $$D' \cnec q'(p'^{-1}(\ol{\nu'(C^\circ)})) = q(p^{-1}(\nu(\ol{\nu'(C^\circ)})))  \subset X,$$ 
which is the union of curves parameterized by $\nu(\ol{\nu'(C^\circ)})$, is a divisor and $D'$ contains $C$.  Finally, note that since $\nu'_{|U}$ is isomorphic onto its image, we have $p'^{-1}(\nu'(C^\circ)) \subset \ti{\nu}(\ti{p}^{-1}(C^\circ))$.  Hence $D' \subset D$.
 \end{proof}

By the above claim, it follows from Theorem~\ref{thm-FLpos} that
\begin{equation*}
  [C] \cdot \ti{q}_*\ti{p}^*\ti{p}_*([\ti{C}] + [C']) \ge [C] \cdot [D] \ge [C] \cdot [D'] > 0,
  \end{equation*} so
\begin{equation}
\begin{split}
\(\ti{p}_*([\ti{C}] + [C'])\)^2  & = 
([\ti{C}] + [C']) \cdot \ti{p}^*\ti{p}_*([\ti{C}] + [C']) \\
& = m\ti{q}^*[C] \cdot \ti{p}^*\ti{p}_*([\ti{C}] + [C'])  = m[C] \cdot \ti{q}_*\ti{p}^*\ti{p}_*([\ti{C}] + [C']) > 0,
\end{split}
\end{equation}
which shows that $\ti{p}_*([\ti{C}] + [C'])$ is big.

Since $\ti{p}$ is bimeromorphic to $p$, $\ti{p}$ is an elliptic fibration.  As $\ti{p}_*([\ti{C}] + [C'])$ is big, it follows from Proposition~\ref{pro-big} that $\ti{\cC}$ is projective, which contradicts the assumption that $a(X) = 2$ because $X$ is bimeromorphic to $\ti{\cC}$. Hence $X$ is projective.
\end{proof}


\newcommand{\etalchar}[1]{$^{#1}$}

\end{document}